\documentclass[11pt,a4paper]{article}

\usepackage{amsmath, amsfonts, amssymb}
\usepackage{theorem}
\usepackage[dvips]{epsfig}
\usepackage{latexsym}
\usepackage{exscale}
\usepackage[latin1]{inputenc}
\newtheorem{thm}{Theorem}[section]
\newtheorem{lem}[thm]{Lemma}
\newtheorem{prop}[thm]{Proposition}

\newtheorem{cor}[thm]{Corollary}
\newtheorem{rem}[thm]{Remark}

\def\be#1\ee{\begin{equation}#1\end{equation}}
\newcommand{\bea}{\begin{eqnarray}}
\newcommand{\eea}{\end{eqnarray}}
\newcommand{\beaa}{\begin{eqnarray*}}
\newcommand{\eeaa}{\end{eqnarray*}}
\newcommand{\bei}{\begin{itemize}}
\newcommand{\eei}{\end{itemize}}
\newcommand{\bee}{\begin{enumerate}}
\newcommand{\eee}{\end{enumerate}}

\def\vp{\varphi}

\def\norm#1{\left\|#1\right\|}             %norm
\def\abs#1{\left\vert #1 \right\vert}      %absolute value

\def\set#1{\left\{#1\right\}}

\def\P{{\mathbb{P}}}
\def\pr#1{\P\left(#1\right)}

\def\R{\mathbb{R}}
\def\E{\mathbb{E}}
\def\Z{{\mathbb Z}}
\def\N{{\mathbb N}}
\def\d{\, \mathrm{d}}
\def\al{\alpha}

\def\dw{E_W}

\newcommand{\eps}{\varepsilon}
\def\E{{\mathbb{E}}\,}
\def\ga{{\gamma}}
\def\H{{{\scriptstyle H}}}

\def\JJ{{I}}

\def\N{\mathbb{N}}
\def\on{{\mathbf 1}}

\def\R{{\mathbb{R}}}

\def\vp{\varphi}
\def\Z{\mathbb{Z}}

\def\pq{\preceq}
\def\pc{\prec}
\def\qp{\succeq}

\newenvironment{proof}[1][] {\noindent {\bf Proof#1:} }{\hspace*{\fill}$\square$\medskip\par}

\def\scp#1#2{\left\langle{#1},{#2}\right\rangle}

\begin{document}

\title{Compactness Properties \\
of Weighted Summation Operators on Trees}
\author{Mikhail Lifshits \and Werner Linde
  }
 \date{\today}
    \maketitle
\begin{abstract}
We investigate compactness properties of weighted summation
operators $V_{\al,\sigma}$ as mapping from $\ell_1(T)$ into $\ell_q(T)$ for some $q\in (1,\infty)$.
Those operators are defined by
$$
(V_{\al,\sigma} x)(t) :=\al(t)\sum_{s\qp t}\sigma(s) x(s)\,,\quad t\in T\;,
$$
where $T$ is a tree with induced partial order $t\pq s$ (or $s\qp t$) for $t,s\in T$.
Here $\al$ and $\sigma$ are given weights on $T$.
We introduce a metric $d$ on $T$ such that compactness properties of $(T,d)$ imply
two--sided estimates for $e_n(V_{\al,\sigma})$, the (dyadic) entropy numbers of $V_{\al,\sigma}$.
The results are applied for concrete trees  as e.g.~moderate increasing, biased or binary trees
and for weights with $\al(t)\sigma(t)$ decreasing either polynomially or exponentially.
We also give some probabilistic applications for Gaussian summation schemes on trees.
\end{abstract}
\bigskip
\bigskip
\bigskip
\bigskip
\bigskip

\vfill
\medskip


\baselineskip=6.0mm
\newpage
\section{Introduction}

This work essentially stems from the article \cite{Lif10} where the entropy of linear
Volterra integral operators was studied in a difficult critical case.
Handling this case required a new technique and it turned out that this technique
could be cleanly elaborated and better explained if we replace  the Volterra operator
by an analogous summation operator on the binary tree. Notice that trees appear naturally
in the study of functional spaces because the Haar base and other
similar wavelet bases indeed have a structure close to that of a binary tree.

The class of summation operators on trees is quite simple and natural but
it is absolutely not investigated and we believe that
a deeper study of its properties, as presented here, is not only interesting in its own right
but might also be helpful as a model for studying more conventional classes of operators.

Thus let $T$ be a finite or infinite tree and let $,\!,\pq"$ be the partial order generated by its structure,
i.e. $t \pq s$  means that $t$ is situated on the way leading from the root of the tree to $s$.
If $k :T\times T\mapsto\R$
is a kernel satisfying
\be
\label{kernel}
\sup_{s\in T}\sum_{t\pq s}\abs{k(t,s)}^q<\infty
\ee
for some $q\in[1,\infty)$, then the Volterra--type summation operator $V$ with
$$
(V x)(t):=\sum_{s\qp t} k(t,s) x(s)\,,\quad t\in T\,,
$$
is bounded from $\ell_1(T)$ into $\ell_q(T)$. Compactness properties of $V$
surely depend on the kernel $k$ as well as on the structure of the underlying tree. It
seems to be hopeless to describe such properties of $V$ in this general context. A first
step could be the investigation of this problem in the case of special kernels $k$ (and for quite general
trees). Thus we restrict ourselves to kernels $k$ which may be written as $k(t,s)=\al(t)\sigma(s)$
for some given weights $\al,\sigma : T\mapsto (0,\infty)$ where we assume that $\sigma$
is non--increasing. Condition (\ref{kernel}) reads then as
\be
\label{alq}
\sup_{s\in T}\left(\sum_{r\pq s}\al(r)^q\right)^{1/q}\sigma(s)<\infty
\ee
and $V=V_{\al,\sigma}$ acts as
\be
\label{Valsig}
(V_{\al,\sigma} x)(t)=\al(t)\,\sum_{s\qp t}\sigma(s) x(s)\,,\quad t\in T\;.
\ee
Note that adding signs to $\alpha$ and $\sigma$ does not change compactness properties (as well as any
other property), thus assuming positive weights we do not lose the generality.

In the linear case $T=\N_0$, those weighted summation operators have been investigated in \cite{CL}. The main
observation in this paper was that such operators may be
regarded as special weighted integration operators and, consequently, their properties follow by
those for integration operators as proved in  \cite{EEH1}, \cite{EEH2}, \cite{LifLi}, \cite{Ma}, etc.

The situation is completely different for general trees. Here an application of known
results about Volterra integration operators is not available. Therefore summation operators in this
general context have to be treated independently and new interesting phenomena appear because the
structure of the underlying tree plays an important role.
\bigskip

The main objective of the present paper is to investigate compactness properties of
operators $V_{\al,\sigma}$ defined in (\ref{Valsig}).
Our basic observation is as follows. Suppose we are given weights $\al$ and
$\sigma$ satisfying (\ref{alq}) with $\sigma$ non--decreasing and let $q\in (1,\infty)$. If $t\pq s$ are in $T$ we define their distance by
$$
d(t,s) :=\max_{t\pc v\pq s}  \left(\sum_{t\pc r\pq v}\al(r)^q\right)^{1/q} \sigma(v)   \;.
$$
Then $d$ may be extended to a metric $d$ on $T$. Let $N(T,d,\eps)$ be the covering numbers of
$(T,d)$, i.e.,
$$
N(T,d,\eps):=\inf\set{n\ge 1 : T=\bigcup_{j=1}^n B_\eps(t_j)}
$$
with (open) $\eps$--balls $B_\eps(t_j)$ for certain $t_j\in T$. We prove that the behavior of
$N(T,d,\eps)$ as $\eps\to 0$ is tightly connected with the degree of compactness of
$V_{\al,\sigma}$. More precisely, let $e_n(V_{\al,\sigma})$ be the sequence of dyadic entropy numbers
defined by
$$
%\label{entrnumbers}
           e_n(V_{\al,\sigma}):=\inf\left[\eps>0 : \set{V_{\al,\sigma} x : \norm{x}_1\le 1}\;
           \mbox{is covered by}\;
           2^{n-1}\;\mbox{open}\; \eps\mbox{--balls in}\;\ell_q(T)\right]
$$
(we refer to \cite{CS} for more information about entropy numbers).
We prove that the upper (lower) bounds for $N(T,d,\eps)$ yield upper
(lower) bounds for $e_n(V_{\al,\sigma})$.

For example, as shown in Theorem \ref{t1}, given $a>0$ and $b\ge 0$ it follows that
$$
N(T,d,\eps)\le c\, \eps^{-a}\abs{\log \eps}^b\quad\mbox{implies}\quad
e_n(V_{\al,\sigma})\le c'\, n^{-1/a-1/p'}(\log n)^{b/a}
$$
with $p:=\min\set{2,q}$ and $1/p'=1-1/p$. In Theorem \ref{t3} we prove a similar result assuming
$N(T,d,\eps)\ge c\, \eps^{-a}\abs{\log \eps}^b$. In particular, if $1<q\le 2$, then
$$
N(T,d,\eps)\approx \eps^{-a}\abs{\log \eps}^b\quad\mbox{implies}\quad
e_n(V_{\al,\sigma})\approx n^{-1/a-1/q'}(\log n)^{b/a}\;.
$$
We also treat the case that $N(T,d,\eps)$
increases exponentially. Besides some critical case, sharp estimates are obtained as well.

Thus in order to get precise estimates for $e_n(V_{\al,\sigma})$ it suffices to describe the behavior
of $N(T,d,\eps)$ in dependence of properties of the weights $\al$ and $\sigma$ and of the structure of the tree. This question is
investigated in Sections \ref{s:uppermod} and \ref{s:lowermod}. Here we prove quite precise estimates
for $N(T,d,\eps)$ in the case of moderate trees (the number of nodes in the $n$--th generation increases
polynomially) or for binary trees provided we know something about the behavior of $\al(t)\sigma(t)$.
In Section \ref{s:bias} we investigate a class of trees where the
branches die out very quickly. Here the behavior of $N(T,d,\eps)$ is completely
different from the one observed for trees where each node has at least one offspring. This example demonstrates
the influence of the tree structure to compactness properties of $V_{\al,\sigma}$.

In Section \ref{s:probab} we sketch a probabilistic interpretation of our results by
providing the asymptotic of small deviation probabilities for some tree--indexed Gaussian
random functions and at the end in Section \ref{s:open} we state some open problems related to the
topic of the present paper.

Let us finally mention that throughout this paper
we always denote by $c$ or $C$ (with or without subscript) universal constants
which may vary even in one line. The constants may depend on $q$ but neither on $n$ nor on
the behavior of the weights.

\section{Trees}
\setcounter{equation}{0}

Let us recall some basic notations related to trees which will be used later on.
In the sequel $T$ always denotes a finite or an infinite
tree. We suppose that $T$ has a unique root which we denote by
$\mathbf 0$ and that each element $t\in T$ has a finite number
$\xi(t)$ of offsprings. Thereby we do not exclude the case $\xi(t)=0$, i.e.,
some elements may "die out". The tree structure leads in natural way
to a partial order $,\!,\pq "$ by letting $t\pq s$ and $s\qp t$ provided there are
$t=t_0, t_1,\ldots, t_m=s$ in $T$ such that for $1\le j\le m$ the
element $t_j$ is an offspring of $t_{j-1}$.
The strict inequalities have the same meaning with the additional assumption $t\not=s$.
Two elements $t,s\in T$ are
said to be comparable provided that either $t\pq s$ or $s\pq t$.
Otherwise we say that $t$ and $s$ are incomparable.

For $t,s\in T$ with $t\pq s$ the order interval $[t,s]$ is defined by
$$
[t,s]:= \set{ r\in T : t\pq r\pq s}
$$
and in a similar way we construct $(t,s]$.

A subset $B\subseteq T$ is said to be a branch provided that all elements in $B$ are
comparable and, moreover, if $t\pq r\pq s$ with $t,s\in B$, then this implies $r\in B$
as well. Of course, finite branches are of the form $[t,s]$ for suitable $t\pq s$.

Given $s\in T$ its order $|s|\ge 0$ is defined by
$$
|s|:=\#\set{t\in T : t \pc s}\;.
$$
Then
$$
%\label{Rn}
R(n):=\#\set{t\in T : |t|=n}\;,\quad n\ge 0\;,
$$
is the number of elements in the $n$--th generation of $T$.

\section{Metrics and $\eps$--Nets on Trees}
\setcounter{equation}{0}

Suppose we are given two weight functions $\al : T\mapsto (0,\infty)$ and $\sigma : T\mapsto (0,\infty)$
where we assume that $\sigma$ is non--increasing, i.e., if $t\pq s$, then it follows that
$\sigma(t)\ge \sigma(s)$.
%Furthermore, we always suppose $\al(t)\not= 0$ for $t\in T$. If
%this is not so, then we delete those $t\in T$ for which $\al(t)=0$ and complete the tree
%structure after that in a natural way.

Given $q\in [1,\infty)$ and $t,v\in T$ with $t\pq v$, we set
$$
\norm{\al\,\on_{(t,v]}}_q := \left(\sum_{t\pc r\pq v} \al(r)^q\right)^{1/q}\;.
$$
Using this, we define a mapping $d : T\times T\mapsto[0,\infty)$
as follows:\\
If $t\pq s$, then we set
\be
\label{metric}
d(t,s) :=\max_{t\pc v\pq s}   \set{\norm{\al\,\on_{(t,v]}}_q\, \sigma(v) }  \;.
\ee
We let $d(t,s):=d(s,t)$ provided that $t\qp s$ and
$$
d(t,s):= d(t\wedge s,t ) + d(t\wedge s,s)
$$
whenever $t$ and $s$ are incomparable. Here $t\wedge s$ denotes the infimum of $t$ and $s$
which may be defined as the maximal element in $[{\mathbf 0},t]\cap[{\mathbf 0},s]$.
\begin{rem}
\rm
Since $\sigma$ is assumed to be non--increasing it follows that for $t\pq s$
$$
d(t,s)=\max_{t\pc v\pq s}\set{\norm{\al\,\on_{(t,v]}}_q\cdot\norm{\sigma\,\on_{[v,s]}}_\infty}\;.
$$
A similar expression (for weights and intervals on $\R$) played an important role in the investigation of
weighted integration operators (c.f.~\cite{EEH1},\cite{EEH2}, \cite{LifLi} and \cite{LLS}).
\end{rem}
\begin{prop}
\label{p1}
The mapping $d$ constructed above is a metric on $T$ possessing the following monotonicity property:
Whenever $t'\pq t\pq s\pq s'$, then $d(t,s)\le d(t',s')$.
\end{prop}

\begin{proof}
The monotonicity property is a direct consequence of the definition of $d$.

Clearly we have $d(t,s)\ge 0$ and since we assumed $\al(t)>0$ for $t\in T$ we see
that $d(t,s)=0$ implies $t=s$. By the construction we also have $d(t,s)=d(s,t)$, thus it remains
to prove the triangular inequality
$$
d(t,s)\le d(t,r) + d(r,s)
$$
whenever $t,s,r\in T$. To verify this inequality one has to treat separately
six different cases in dependence of the relation
between $t,s$ and $r$. Among them only one is non--trivial,
namely, the case that $t,s$ and $r$ are on
a common branch and satisfy $t\pq r\pq s$ or, equivalently, $s\pq r\pq t$. Therefore we only include the
proof in that situation.

Assume $t\pq r\pq s$ and choose $v$ in $T$ with $t\pc v\pq s$
where the maximum in (\ref{metric}) is attained. Then we have to distinguish between the
two following cases: $v\pq r$ and $r\pc v$.

In the first case we have
$$
d(t,s)=\norm{\al\,\on_{(t,v]}}_q\,\sigma(v)\le \max_{t\pc v'\pq r}  \set{\norm{\al\,\on_{(t,v']}}_q\,\sigma(v')
}= d(t,r)\le d(t,r) + d(r,s),
$$
and we are done.

Suppose now $r\pc v$. Here we argue as follows:
$$
d(t,s)=\norm{\al\,\on_{(t,v]}}_q\,\sigma(v)\le \left(\norm{\al\,\on_{(t,r]}}_q +
\norm{\al\,\on_{(r,v]}}_q\right)\,\sigma(v),
$$
and since $\sigma$ is non--increasing, it follows that
$$
d(t,s)\le \norm{\al\,\on_{(t,r]}}_q\,\sigma(r) + \norm{\al\,\on_{(r,v]}}_q\,\sigma(v)
\le d(t,r) + d(r,s),
$$
as asserted.
Thus the proof is completed.
\end{proof}
\bigskip

Our next objective is to investigate $\eps$--nets for $T$ w.r.t.~the metric $d$ possessing
an additional useful property.
Given $\eps >0$, a set $S\subseteq T$ is said to be an \textit{order} $\eps$--net provided that for
each $t\in T$ there is a $s \in S$ satisfying $d(s,t)<\eps$ as well as $s\pq t$. The corresponding
\textit{order} covering numbers of $T$ are then
\be
\label{order}
\tilde N(T,d,\eps):=\inf\set{\# S : S\;\mbox{is an order $\eps$--net of}\; T}\;.
\ee
Recall that the usual covering numbers $N(T,d,\eps)$ were defined by
$$
N(T,d,\eps)=\inf\set{\# S : S\subset T,\; T=\bigcup_{s\in S} B_\eps(s)}
$$
where $B_\eps(s)$ is the open $\eps$--ball centered at $s\in T$, i.e.
$$
B_\eps(s):=\set{r\in T : d(r,s)<\eps}\;.
$$
Clearly we have
$$
N(T,d,\eps)\le \tilde N(T,d,\eps)\;,
$$
but as we shall see now, a slightly weaker reverse estimate is valid as well. More precisely we have
the following.

\begin{prop}
\label{p1a}
Let $d$ be the metric defined in $(\ref{metric})$. Then for any $\eps>0$, it is true that
\be
\label{entrest}
\tilde N(T,d,2\eps)\le N(T,d,\eps)\;.
\ee
\end{prop}

\begin{proof}
Take any $\eps$--net $S\subset T$. For each $s\in S$ we choose
$r_s\in B_{\eps}(s)$ such that
$r_s\wedge s$ is the minimal element in $\set{r\wedge s : r\in B_{\eps}(s)}$.
Then we have
$$
r_s\wedge s \pq r\wedge s\pq s\;,
$$
whenever $r\in B_{\eps}(s)$.

Set
$$
\tilde S :=\set{r_s\wedge s : s\in S}\;.
$$
Clearly, we have $\# \tilde S \le \# S$, hence it suffices to prove that $\tilde S$ is an order
$2\eps$--net of $T$. To this end take any $t\in T$. Then there is an $s\in S$ such that
$t\in B_\eps(s)$ and by the choice of $r_s$ it follows that $r_s\wedge s \pq t\wedge s \pq t$.
Thus it remains to estimate
the distance between $r_s\wedge s$ and $t$.
Note that the definition of $d$ implies $d(r_s\wedge s,s)\le d(r_s,s)<\eps$.
Thus the triangle inequality leads to
$$
d(r_s\wedge s, t)\le d(r_s\wedge s,s) + d(t,s)< 2\eps
$$
because of $t\in B_\eps(s)$. This completes the proof.
\end{proof}

\section{Upper Entropy Estimates for Weighted Summation Operators}
\label{s:upper1}
\setcounter{equation}{0}
Here and later on the basic assumption about the weight functions $\al$ and $\sigma$
is that they satisfy (\ref{alq})
for some fixed $q\in (1,\infty)$ and that $\sigma$ is non--increasing.

In a first step we investigate weights $\sigma$
attaining only values in $\set{2^{-m} : m\in\Z}$. Without losing generality assume
$\sigma(\mathbf 0)=1$, hence there are subsets $I_m\subseteq T$, $m\ge 0$, such that
\be
\label{sigma}
\sigma(s)=\sum_{m=0}^\infty 2^{-m}\,\on_{I_m}(s)\,,\quad s\in T\;.
\ee
Since $\sigma$ is supposed to be non--increasing, the sets $I_m$ possess the following properties:
\bee
\item
It holds $T=\bigcup_{m=0}^\infty I_m$ and $I_l\cap I_m=\emptyset$ provided that $l\not= m$.
\item
Whenever $B\subseteq T$ is a branch, then for each $m\ge 0$ either $B\cap I_m=\emptyset$
or it is an order interval in $T$. Furthermore, if $l< m$, $t\in B\cap I_l$, $s\in B\cap I_m$,
then this implies $t\pc s$.
\eee
Define an operator $W$ on $\ell_1(T)$ by
\be
\label{defW}
(W x)(t) := \al(t)\,\sum_{s\succeq t\atop{s\in I_m}}\sigma(s) x(s)=
\al(t)\,2^{-m}\,\sum_{s\succeq t\atop{s\in I_m}}x(s)\;,\quad t\in I_m\;.
\ee
The mapping $W$ acts as a "partial" weighted summation
operator depending on the partition $(I_m)_{m\ge 0}$.
We claim that condition (\ref{alq}) implies that $W$ is a bounded operator from $\ell_1(T)$ into
$\ell_q(T)$. To see this
define the unit vectors $\delta_t\in \ell_1(T)$, $t\in T$,  by
$$
\delta_t(r):=
\left\{
\begin{array}{ccl}
1&:& r=t\\
0&:& r\not= t\;.
\end{array}
\right.
$$
Then
$$
%\label{Wdelta}
W(\delta_t) = \sum_{r\pq t\atop{r\in I_m}}\al(r)\,\sigma(r)\,\delta_r=
2^{-m}\,\sum_{r\pq t\atop{r\in I_m}}\al(r)\,\delta_r\,,\quad t\in I_m\;,
$$
hence (\ref{alq}) implies $\sup_{t\in T}\norm{W(\delta_t)}_q<\infty$ and
$W :\ell_1(T)\mapsto\ell_q(T)$ is well--defined and bounded.
\medskip

Define the set $\dw \subseteq \ell_q(T)$ by
$$
%\label{deltaW}
\dw :=\set{ W(\delta_t) : t\in T}
$$
and let the metric $d$ on $T$ be as in (\ref{metric}) with weights $\al$ and $\sigma$
satisfying (\ref{alq}) and (\ref{sigma}), respectively. Then the following holds.

\begin{prop}
\label{p3}
We have
$$
N(\dw,\|\cdot\|_q,\eps)\le \tilde N(T,d,\eps) +1\;.
$$
\end{prop}

\begin{proof}
Fix $\eps>0$ and choose an arbitrary order $\eps$--net $S$ in $T$ (w.r.t.~the metric $d$).
Given $t\in T$, there is a unique $m\ge 0$ with $t\in I_m$. By definition we find an $s\in S$
satisfying $d(s,t)<\eps$ as well as $s\pq t$. Assume first that $s\in I_m$ as well. Then we get
$$
\norm{W(\delta_t)-W(\delta_s)}_q = \left(\sum_{s\pc r\pq t}  \al(r)^q\right)^{1/q}\cdot\,2^{-m}
=\norm{\al\,\on_{(s,t]}}_q\,\sigma(t)\le d(s,t)<\eps\;.
$$
Otherwise, if $s\in I_l$ for a certain $l<m$, we argue as follows:
$$
%\label{estW}
\norm{W(\delta_t)}_q=\left(\sum_{r\pq t\atop{r\in I_m}} \al(r)^q\right)^{1/q}\cdot 2^{-m}\le
\left(\sum_{s\pc r\pq t} \al(r)^q\right)^{1/q}\cdot\sigma(t)\le d(s,t)<\eps\;.
$$
Consequently, the set
$$
\set{W(\delta_s) : s\in S}\cup\{0\}
$$
is an $\eps$--net  of $\dw$ in $\ell_q(T)$. This being true for any order net $S$ completes
the proof.
\end{proof}

\begin{prop}
\label{p4}
For $q\in (1,\infty)$ let $p:=\min\set{q,2}$ and $1/p':= 1- 1/p$. Furthermore let $a>0$ and
$0\le b<\infty$ .
If we have
\be
\label{tilN}
\tilde N(T,d,\eps)\le c\,\eps^{-a}|\log \eps|^b\;,
\ee
then this implies
\be
\label{enW}
e_n(W : \ell_1(T)\mapsto\ell_q(T))\le c'\,n^{-1/a-1/p'} (\log n)^{b/a}\;.
\ee
If instead of $(\ref{tilN})$ we  only have
\be
\label{logN}
\log \tilde N(T,d,\eps)\le c\, \eps^{-a}\;,
\ee
then we get
\be
\label{enW1a}
e_n(W : \ell_1(T)\mapsto\ell_q(T))\le c'\, n^{-1/p'}(\log n)^{1/p' -1/a}
\ee
whenever $a<p'$ while for $a>p'$ we have
\be
\label{enW2}
e_n(W : \ell_1(T)\mapsto\ell_q(T))\le c' n^{-1/a}.
\ee
\end{prop}

\begin{proof}
If we assume (\ref{tilN}), by Proposition \ref{p3} we also have
\be
\label{Ndw}
N(\dw, \|\cdot\|_q,\eps)\le c\, \eps^{-a}|\log \eps|^b\;.
\ee
Observe that $e_n(W)= e_n(\mathrm{aco}(\dw))$, where $\mathrm{aco}(B)$ denotes
the absolutely convex hull of a set $B\subseteq \ell_q(T)$.
Thus we may use known estimates for the entropy of absolutely convex hulls as can be found in \cite{CKP} or
\cite{St}.
For example, assuming (\ref{Ndw}) we may apply Corollary 5 in \cite{St}. Recall that $\ell_q(T)$ is of type $p$ with $p=\min\set{2,q}$.
Hence we get
$$
e_n(W)=e_n(\mathrm{aco}(\dw))\le c'\,n^{-1/a-1/p'} (\log n)^{b/a}
$$
which completes the proof of (\ref{enW}).

Assuming (\ref{logN}) estimates (\ref{enW1a}) and (\ref{enW2}) follow by similar arguments using
Corollaries 4 and 3 in \cite{St}, respectively.
\end{proof}
\medskip

Our next objective is to apply the previous results to weighted summation operators. To this end
let $\al$ and $\sigma$ be two weight functions satisfying (\ref{alq}). Here $\sigma$
is an arbitrary non--increasing weight.
Then we  define the weighted summation operator $V_{\al,\sigma}$ as in (\ref{Valsig}).
Under the assumptions on the weights the operator $V_{\al,\sigma}$ is well--defined
and bounded from $\ell_1(T)$ into $\ell_q(T)$.

The main goal is to relate the degree of compactness of $V_{\al,\sigma}$ with the behavior of
$\tilde N(T,d,\eps)$ as $\eps\to 0$.
Here the metric $d$ is defined as in (\ref{metric}) by $\al$ and $\sigma$. In a first step we suppose
that $\sigma$ is of the special form (\ref{sigma}) with sets $I_m\subseteq T$ defined there.

Given $t\in T$, set
$$
K_t:=\set{k\ge 0 : I_k\cap[\mathbf 0,t]\not=\emptyset}\;.
$$
Consequently, if $k\in K_t$, then $I_k\cap[\mathbf 0,t]=[\lambda_k(t),\theta_k(t)]$
for some $\lambda_k(t)\pq \theta_k(t)\pq t$. Note that $\theta_m(t)=t$ for $t\in I_m$ and
\be
\label{intervals}
[\mathbf 0,t]=\bigcup_{k\in K_t}[\lambda_k(t),\theta_k(t)] .
\ee
Define now an operator $Z: \ell_1(T)\mapsto\ell_1(T)$ by
\be
\label{defZ}
Z(\delta_t):= \sum_{k\in K_t} 2^{k-m}\delta_{\theta_k(t)}\;,\quad t\in I_m\;.
\ee

\begin{prop}
\label{p5}
Assume $(\ref{alq})$ and $(\ref{sigma})$ and define
$W: \ell_1(T)\mapsto\ell_q(T)$ and  $Z: \ell_1(T)\mapsto\ell_1(T)$ as in $(\ref{defW})$
and $(\ref{defZ})$, respectively. Then $Z$ is bounded
with $\norm{Z}\le 2$ and, moreover, the operator
$V_{\al,\sigma}$ given by $(\ref{Valsig})$ admits a decomposition
\be
\label{VWZ}
V_{\al,\sigma}=W\circ Z.
\ee

\end{prop}

\begin{proof}
By the construction, for each $t\in I_m$ we have
$$
\norm{Z(\delta_t)}_1 \le \sum_{k\in K_t} 2^{k-m}\le \sum_{k=0}^{m} 2^{k-m}\le 2\;,
$$
hence, in view of $\norm{Z}=\sup_{t\in T}\norm{Z(\delta_t)}_1$ this implies
$\norm{Z}\le 2$ as asserted.

To prove (\ref{VWZ}) first note that for $t\in T$ and $k\in K_t$ we get
$$
W(\delta_{\theta_k(t)})=\sigma(\theta_k(t))\sum_{r\in [\lambda_k(t),\theta_k(t)]} \al(r)\delta_r
  = 2^{-k}\sum_{r\in [\lambda_k(t),\theta_k(t)]} \al(r)\delta_r\;,
$$
hence, if $t\in I_m$, then by (\ref{intervals}) this implies
$$
W(Z(\delta_t))=
\sum_{k\in K_t} 2^{k-m}\left[2^{-k}\sum_{r\in [\lambda_k(t),\theta_k(t)]} \al(r)\delta_r\right]
=2^{-m} \,\sum_{r\in[0,t]}\al(r)\delta_r\;.
$$
On the other hand,
$$
V_{\al,\sigma}(\delta_t)=\sigma(t)\,\sum_{r\pq t}\al(r)\delta_r = 2^{-m} \,\sum_{r\in[0,t]}\al(r)\delta_r \;,
$$
and it follows that
$$
V_{\al,\sigma}(\delta_t)=W(Z(\delta_t))\;.
$$
This being true for any $t\in T$ proves (\ref{VWZ}).
\end{proof}

\begin{thm}
\label{t1}
Let $\al$ and $\sigma$ be weight functions satisfying $(\ref{alq})$ where $\sigma$ is arbitrary
non--increasing weight.
If
$$
\tilde N(T,d,\eps)\le c\, \eps^{-a}\abs{\log\eps}^b
$$
for some $a>0$ and $b\ge 0$, then this implies
$$
e_n(V_{\al,\sigma} :\ell_1(T)\mapsto \ell_q(T))\le c\,
 n^{-1/a-1/p'} (\log n)^{b/a}
$$
with $p$ as in Proposition $\ref{p4}$.
If
\be
\label{logtil}
\log \tilde N(T,d,\eps)\le c\, \eps^{-a}
\ee
we get
$$
e_n(V_{\al,\sigma} : \ell_1(T)\mapsto\ell_q(T))\le c'\, n^{-1/p'}(\log n)^{1/p' -1/a}
$$
whenever $a<p'$ while for $a>p'$ we have
$$
e_n(V_{\al,\sigma} : \ell_1(T)\mapsto\ell_q(T))\le c' n^{-1/a}.
$$
\end{thm}

\begin{proof}
Suppose as before $\sigma(\mathbf 0)=1$ and for $m\ge 0$ define subsets $I_m\subseteq T$ by
$$
%\label{Im}
I_m:=\set{t\in T : 2^{-m-1}<\sigma(t)\le 2^{-m}}\;.
$$
If
$$
\hat\sigma(t):=\sum_{m\ge 0} 2^{-m}\on_{I_m}(t)\,,\quad t\in T\,,
$$
then $\hat\sigma$ is a non--increasing weight function as in (\ref{sigma}).
By the construction it follows that
\be
\label{hats}
\sigma(t)\le \hat\sigma(t)\le 2\,\sigma(t)\,,\quad t\in T\;.
\ee
Define the metrics $d$ and $\hat d$ as in (\ref{metric}) by $\al$ and by $\sigma$ or
$\hat\sigma$, respectively. In view of (\ref{hats}) we get
$$
d(t,s)\le \hat d(t,s)\le 2\,d(t,s)\,,\quad  t,s\in T\;,
$$
hence
$$
\tilde N(T,\hat d,\eps/2)\le \tilde N(T,d,\eps)
$$
which implies
$$
\tilde N(T,\hat d,\eps)\le c\, \eps^{-a}\abs{\log\eps}^b
$$
as well. But now we are in the situation of Proposition \ref{p4} and obtain
\be
\label{enW1}
e_n(W : \ell_1(T)\mapsto\ell_q(T))\le c\, n^{-1/a-1/p'} (\log n)^{b/a}\;.
\ee
An application of Proposition \ref{p5} yields now
$$
e_n(V_{\al,\hat \sigma})=e_n(W\circ Z)\le e_n(W)\|Z\|\le 2\,e_n(W)\,,
$$
hence by (\ref{enW1}) it follows that also
$$
e_n(V_{\al,\hat \sigma})\le c\, n^{-1/a-1/p'} (\log n)^{b/a}\;.
$$
To complete the proof note that (\ref{hats}) implies
that the diagonal operator  $\Delta :\ell_1(T)\mapsto\ell_1(T)$
defined by
$$
\Delta(\delta_t):=\frac{\sigma(t)}{\hat\sigma(t)}\,\delta_t
$$
is bounded with $\norm{\Delta}\le 1$.  Of course,
$$
V_{\al,\sigma}=V_{\al,\hat \sigma}\circ\Delta\;,
$$
hence
\be
\label{enVal}
e_n(V_{\al,\sigma})\le e_n(V_{\al,\hat \sigma})\,\norm{\Delta}\le e_n(V_{\al,\hat \sigma})
\ee
completing the proof of the first part.

The second part is proved by exactly the same arguments. Indeed, (\ref{logtil}) implies
$$
\log \tilde N(T,\hat d,\eps)\le c\,\eps^{-a}\;.
$$
An application of Proposition \ref{p5} yields now
$$
e_n(V_{\al,\hat \sigma})=e_n(W\circ Z)\le e_n(W)\|Z\|\le 2\,e_n(W)
$$
and the estimates follow by the second part of Proposition \ref{p4} via (\ref{enVal}).
\end{proof}

\begin{remark}
\rm
The critical case $a=p'$ is excluded in the second part of Theorem \ref{t1}.
This is due to the fact that in that case only weaker estimates for $e_n(\mathrm{aco}(\dw))$,
hence for $e_n(W)$ and also for $e_n(V_{\al,\sigma})$ are available. Indeed, using
Corollary 1.4 in \cite{CSt} it follows that (\ref{logtil}) only gives
$$
e_n(V_{\al,\sigma} :\ell_1(T)\mapsto\ell_q(T))\le c'\,n^{-1/a}\,\log n\;.
$$
But the results in \cite{Lif10} suggest that the right order in that case is $n^{-1/a}$, i.e.,
the above estimate probably contains an unnecessary extra $\log$--term.
\end{remark}

\section{Lower Entropy Estimates}
\setcounter{equation}{0}

We start with a quite general lower estimate for weighted summation operators on trees.

\begin{prop}
\label{p6}
Suppose there are $m$ pairs of elements $t_i,s_i$ in $T$ possessing the following properties.
\bee
\item
It holds $t_i\pc s_i$ and $(t_i,s_i]\cap (t_j,s_j]=\emptyset$ for $1\le i,j\le m$, $i\not=j$.
\item
For some $\eps>0$ we have $d(t_i,s_i)\ge \eps$, $1\le i\le m$.
\eee
Then this implies
$$
e_n(V_{\al,\sigma}:\ell_1(T)\mapsto\ell_q(T))\ge c\,\eps\left(\frac{\log(1+\,m/n)}{n}\right)^{1/q'}
$$
with some $c=c(q)$
whenever $\log m\le n\le m$.
\end{prop}

\begin{proof} The strategy of the following construction consists of "inscribing"
the well studied identity operator from $\ell_1^m$ into $\ell_q^m$ into our
operator $V_{\al,\sigma}$.

The definition of the metric $d$ implies the existence of $v_i\in T$ such that
$t_i\pc v_i\pq s_i$ and
$$
%\label{vi}
\left(\sum_{t_i\pc r\pq v_i}\al(r)^q\right)^{1/q}
\sigma(v_i)   \ge \eps\,,\quad 1\le i\le m\;.
$$
By assumption the intervals $J_i:=(t_i,v_i]$, $1\le i\le m$, are disjoint subsets of $T$.

Next define elements $y_i\in\ell_1(T)$ by
$$
y_i:= \delta_{v_i}-\frac{\sigma(v_i)}{\sigma(t_i)}\,\delta_{t_i}\,,\quad 1\le i\le m\;,
$$
as well as an operator $\JJ : \ell_1^m\mapsto \ell_1(T)$ by setting
$$
\JJ(\delta_i):= y_i\,,\quad 1\le i\le m\;.
$$
Here $\delta_i$ is the $i$--th unit vector in $\ell_1^m=\ell_1(\set{1,\ldots,m})$.
Then $\sigma(v_i)\le \sigma(t_i)$ implies $\norm{y_i}_1\le 2$, hence $\norm{\JJ}\le 2$ as well.

The image $z_i\in\ell_q(T)$ of $y_i$ w.r.t.~$V_{\al,\sigma}$ equals
$$
z_i:=V_{\al,\sigma}(y_i)= \sigma(v_i)\sum_{t_i\pc r\pq v_i}\al(r)\,\delta_r\,,\quad 1\le i\le m\;.
$$
In particular, the support of $z_i$ is contained in $J_i$.

Finally, let
$$
\beta_i := \left(\sum_{t_i\pc r\pq v_i}\al(r)^q\right)^{1/q'}
$$
and
$$
b_i := \beta_i^{-1}\sum_{t_i\pc r\pq v_i}\al(r)^{q-1}\,\delta_r\;.
$$
By the choice of the $\beta_i$ we obtain $\norm{b_i}_{q'}=1$. Moreover, since the order
intervals $J_i$ are disjoint, it follows that
$$
\scp{z_i}{b_j}=0\,,\quad 1\le i,j\le m\;, i\not= j\,,
$$
while
\be
\label{zibi}
\scp{z_i}{b_i}= \sigma(v_i)\left(\sum_{t_i\pc r\pq v_i}\al(r)^q\right)^{1/q}\ge \eps\,,\quad
1\le i\le m\;.
\ee
If $P :\ell_q(T)\mapsto \ell_q^m$ is given by
$$
P(z) := \left(\scp{z}{b_i}\right)_{i=1}^m\,,\quad z\in \ell_q(T)\,,
$$
it holds $\|P\|\le 1$. Indeed, if $z\in \ell_q(T)$, then it follows that
$$
\norm{P(z)}_q^q =\sum_{i=1}^m \abs{\scp{z}{b_i}}^q =\sum_{i=1}^m \abs{\scp{z}{b_i\,\on_{J_i}}}^q
\le \sum_{i=1}^m \norm{z\,\on_{J_i}}_q^q \,\norm{b_i}_{q'}^q\le \norm{z}_q^q
$$
as claimed.

Summing up,
$$
P V_{\al,\sigma} \JJ(\delta_i) =\scp{z_i}{b_i}\delta_i\,,\quad 1\le i\le m\,,
$$
and because of (\ref{zibi}) we obtain for the identity $\mathrm{Id}_m$ from $\ell_1^m$ into $\ell_q^m$
$$
\mathrm{Id}_m = \Delta \circ(P V_{\al,\sigma} \JJ)
$$
with a diagonal operator $\Delta$ satisfying
$\norm{\Delta :\ell_q^m\mapsto\ell_q^m}\le \eps^{-1}$.
Consequently, we arrive at
$$
e_n(\mathrm{Id}_m : \ell_1^m\mapsto\ell_q^m)\le \eps^{-1} e_n(P V_{\al,\sigma} \JJ)
\le 2\,\eps^{-1}\,e_n(V_{\al,\sigma})\;.
$$
To complete the proof note that a result of Sch\"utt (cf. \cite{Sch}) asserts that
$$
e_n(\mathrm{Id}_m : \ell_1^m\mapsto\ell_q^m)\ge c \left(\frac{\log(1+\,m/n)}{n}\right)^{1/q'}
$$
as long as $\log m\le n\le m$.
\end{proof}

In order to apply Proposition \ref{p6} we have to find sufficiently many order
intervals $(t_i,s_i]$ possessing the properties stated above. The next result shows that
we can find at least $N(T,d,2\eps)-1$ such intervals.

\begin{prop}
\label{p6a}
Let $\eps>0$ be given. Then there are at least $N(T,d,2\eps)-1$ order intervals $(t_i,s_i]$
such that $d(t_i,s_i)\ge \eps$ and $(t_i,s_i]\cap(t_j,s_j]=\emptyset$ provided that $i\not=j$.
\end{prop}

\begin{proof}
Let $S=\{s_1,\ldots,s_n\}\subset T$ be a maximal $2\eps$--distant set, i.e., it holds $d(s_i,s_j)\ge 2\eps$
whenever $i\not=j$. Since $S$ is chosen maximal, for any $t\in T$ there is an $s_i\in S$
with $d(t,s_i)<2\eps$. Thus $S$ is a $2\eps$--net and, consequently, it follows that
$n\ge N(T,d,2\eps)$. Among all elements in $S$ there is at most one $s_i$ with $d(\mathbf 0,s_i)<\eps$.
Thus, by changing the numeration we may assume $d(\mathbf 0, s_j)\ge \eps$ for $1\le j\le n-1$.
For each such $s_j\in S$ we define now $t_j\pc s_j$ as follows: It holds $d(t_j,s_j)\ge \eps$,
but if $t_j\pc t\pq s_j$, then $d(t,s_j)<\eps$. Such $t_j$ exist (and are uniquely
determined) by $d(\mathbf 0,s_j)\ge \eps$ and by the monotonicity property of $d$. We claim now
that the order intervals $(t_1,s_1],\ldots,(t_{n-1},s_{n-1}]$ possess the desired properties. By the
construction $d(t_j,s_j)\ge \eps$ and it remains to prove that the intervals are disjoint. Assume to the contrary
that there is some $t$ in $(t_i,s_i]\cap(t_j,s_j)$ for certain $i\not=j$. Then it follows
$d(t,s_i)<\eps$ as well as $d(t,s_j)<\eps$ by the choice of the $t_j$. This implies
$d(s_i,s_j)\le d(t,s_i) + d(t,s_j)<2\eps$ which contradicts the choice of the set $S$ and completes the
proof.
\end{proof}

Let us state a first consequence of Propositions \ref{p6} and \ref{p6a}.

\begin{thm}
\label{t3}
Suppose that for some $a>0$ and $b\ge 0$ we have
\be
\label{NTd1}
N(T,d,\eps)\ge c\,\eps^{-a}\abs{\log \eps}^b\,.
\ee
Then this implies
\be
\label{lowent}
e_n(V_{\al,\sigma}: \ell_1(T)\mapsto\ell_q(T))\ge \tilde c\,n^{-1/a-1/q'}(\log n)^{b/a}
\ee
with a constant $\tilde c=\tilde c(c,q)$.
In particular, if $1<q\le 2$, then
\be
\label{NTd2}
N(T,d,\eps)\approx\eps^{-a}\abs{\log \eps}^b
\ee
implies
$$
e_n(V_{\al,\sigma}: \ell_1(T)\mapsto\ell_q(T))\approx\tilde c\,n^{-1/a-1/q'}(\log n)^{b/a}\;.
$$
\end{thm}
\begin{proof}
In view of Proposition \ref{p6a} the assumption implies that there are $m$ disjoint order intervals
$(t_i,,s_i]$ with $d(t_i,s_i)\ge \eps$. Hereby we may choose $m$ of order $\eps^{-a}\abs{\log \eps}^b$.
Next we apply Proposition \ref{p6} with $n=m$ and obtain
$$
e_n(V_{\al,\sigma})\ge c\,\eps\left(\frac{\log 2}{n}\right)^{1/q'}\ge \tilde c\,
n^{-1/a-1/q'}(\log n)^{b/a}\;.
$$
This completes the proof.
\end{proof}

\begin{remark}
\rm
~\\
(1) Note that by (\ref{entrest}) in (\ref{NTd1}) as well as in
(\ref{NTd2}) the covering numbers $N(T,d,\eps)$ may be
replaced by the order numbers $\tilde N(T,d,\eps)$.

\noindent
(2)
It remains open whether or not in (\ref{lowent})
the expression $n^{-1/a-1/q'}$ may be replaced by $n^{-1/a-1/2}$ whenever $2<q<\infty$. For
those $q$ remains a gap between the upper estimate in Theorem \ref{t1} and the lower one in
Theorem \ref{t3}.
\end{remark}
\bigskip

Our next objective is an application of Propositions \ref{p6} and \ref{p6a} in the case of
rapidly increasing covering numbers.

\begin{thm}
\label{t3a}
Suppose that
$$
\log N(T,d,\eps)\ge c\,\eps^{-a}
$$
for a certain $a>0$. Then this implies
\be
\label{exp1}
e_n(V_{\al,\sigma} :\ell_1(T)\mapsto \ell_q(T))\ge c\, n^{-1/q'}(\log n)^{1/q'-1/a}
\ee
provided that $a<q'$. On the other hand, if $q'\le a$, then it follows that
\be
\label{exp2}
e_n(V_{\al,\sigma} :\ell_1(T)\mapsto \ell_q(T))\ge c\, n^{-1/a}\;.
\ee
\end{thm}
\begin{proof}
First observe that Proposition \ref{p6a} implies the existence of $m$ disjoint order intervals
$(t_i,s_i]$ with $d(t_i,s_i)\ge \eps$ where the number $m$ satisfies $\log m\approx \eps^{-a}$.

So let us prove (\ref{exp1}). We use Proposition \ref{p6} with $n\approx \sqrt m$ and note that
the choice of $m$ and $n$ implies
$\eps\approx (\log m)^{-1/a}\approx (\log n)^{-1/a}$.  Of course,
$\log m\le n\le m$, thus Proposition \ref{p6} applies and leads to
\beaa
e_n(V_{\al,\sigma}) &\ge& c\,\eps\left(\frac{\log(1+\sqrt m)}{n}\right)^{1/q'}
\ge c\,(\log m)^{-1/a+1/q'}\,n^{-1/q'}\\
&\ge& c\,(\log n)^{-1/a+1/q'}\,n^{-1/q'}
\eeaa
as asserted.

Inequality (\ref{exp2}) follows by similar arguments. The number $m$ is chosen as before but this time
we take $n$ of order $\log m$. This implies $\eps\approx n^{-1/a}$ and we get
$$
e_n(V_{\al,\sigma}) \ge c\,\eps\left(\frac{\log(1+\frac{m}{\log m})}{n}\right)^{1/q'}
\ge c\,\eps\ge c\,n^{-1/a}
$$
as asserted.
\end{proof}
\begin{remark}
\rm
Note that (\ref{exp1}) as well as (\ref{exp2}) are valid for all $a>0$. But for $a\le q'$ the
first estimate is better while for $a\ge q'$ the second one leads to a better lower bound.
\end{remark}

\section{Examples of Upper Entropy Estimates}
\label{s:uppermod}
\setcounter{equation}{0}
The  aim of this section is to apply the previous results for weights and trees
satisfying
certain growth assumptions.
We start with
assuming that there is
a strictly decreasing, continuous function $\vp$ on $(0,\infty)$ with
$$
\int_0^\infty\vp(x)\d x<\infty
$$
such that for some fixed $q<\infty$ holds
\be
\label{alsivp}
(\al(t)\sigma(t))^q\le \vp(|t|)\;, \quad t\in T\,.
\ee
The next objective is to construct order $\eps$--nets on $\N$ for a metric generated by $\vp$.
Later on those nets on $\N$ lead in natural way to nets on trees.
Given $\vp$ as above define $\Phi$ on $[0,\infty]$ by
\be
\label{Phi}
\Phi(y):=\int_y^\infty \vp(x)\d x\,, \quad  0\le y<\infty\,,
\ee
and
$\Phi(\infty):=0$. The generated metric $\bar d$ on $[0,\infty]$ is then defined by
\be
\label{dbar}
\bar d(y_1,y_2):= \Phi(y_1)-\Phi(y_2)=\int_{y_1}^{y_2}\vp(x)\d x
\ee
provided that $y_1\le y_2$. Given $\eps>0$ we construct a $2\, \eps$--net
for $(\N,\bar d)$ as follows. First we take all points in $\N$
up to the level $\vp^{-1}(\eps)$, i.e., as a first part of the net we choose
$$
M_\eps:=\set{n\ge 1 : n\le \vp^{-1}(\eps)}=\set{n\ge 1 : \vp(n)\ge \eps}
$$
and note that $\# M_\eps\le  \vp^{-1}(\eps)$.

It remains to find a suitable $2\eps$--cover for $\set{n\ge 1 : n\ge \vp^{-1}(\eps)}$.
Here we proceed as follows. For $k=1,\ldots,N$ set
\be
\label{uktilde}
\tilde u_k:=\Phi^{-1}(k\,\eps)
\ee
where the number $N$ is chosen as
\beaa
%\label{defN}
\nonumber
N:=\max\set{k\ge 1 : \tilde u_k\ge \vp^{-1}(\eps)}&=&\max\set{k\ge 1 : k\,\eps\le \Phi(\vp^{-1}(\eps))}\\
&=& \max\set{k\ge 1 : k\le \frac{\Phi(\vp^{-1}(\eps))}{\eps}}\;.
\eeaa
Note that $\tilde u_1>\tilde u_2>\cdots>\tilde u_N$ and, moreover, since in that region $\vp(x)<\eps$
we necessarily have $\tilde u_{k-1}-\tilde u_k> 1$, $k=1,\ldots, N$. Hence, setting (here
$[u]$ denotes the integer part of $u\in\R$)
$$
u_k:= [\tilde u_k]\,,\quad k=1,\ldots,N-1\;,
$$
it follows that $\tilde u_1\ge u_1> \tilde u_2\ge\cdots \ge u_{N-1}>\tilde u_N$. It remains to
define $u_N$. If $[\tilde u_N]\ge \vp^{-1}(\eps)$ we set $u_N:=[\tilde u_N]$. Otherwise we take $u_N:=[\vp^{-1}(\eps)]+1$.
By the construction it follows that $\bar d(u_k,m)<2\,\eps$ for all $m\in \N$ with
$u_k\le m<u_{k-1}$ where $u_0:=\infty$.
Consequently, the set
\be
\label{Deps0}
\bar S_\eps:=M_\eps\cup\set{u_1,\ldots,u_N}
\ee
is a $2\,\eps$--net of $(\N,\bar d)$.

Next we want to apply the preceding construction to build suitable $\eps$--nets on trees.
Recall that
$R(n)$ denotes the number of elements in the $n$--th generation of a tree.

\begin{prop}
\label{p7}
Let $T$ be a tree such that $R(n)\le \rho(n)$ for a certain continuous, non--decreasing function
$\rho$ on $[0,\infty)$. Furthermore, suppose that the weights $\al$ and $\sigma$ on $T$
satisfy $(\ref{alsivp})$ for a certain $q\ge 1$ and some function $\vp$ as before. Define the metric $d$ as in
$(\ref{metric})$ with $\al$, $\sigma$ and $q$. Then it follows
$$
\tilde N(T,d,\eps)\le
\int_0^{\vp^{-1}(\eps^q/2)+1}\rho(x)\d x
+ \rho(\Phi^{-1}(\eps^q/2))+
2\,\eps^{-q}\,
\int_{\vp^{-1}(\eps^q/2)}^{\Phi^{-1}(\eps^q/2)}
 \rho(y)\vp(y)\d y
$$
where $\Phi$ is as in $(\ref{Phi})$.
\end{prop}

\begin{proof}
Assuming (\ref{alsivp}) it follows (recall that $\sigma$ is non--increasing) that
for all $t\preceq s$ in $T$
\beaa
d(t,s) &=&\max_{t\pc v\pq s}\Bigg\{\Big(\sum_{t\pc r\pq v}\al(r)^q\Big)^{1/q}\sigma (v)\Bigg\}
\le \Big(\sum_{t\pc r\pq s}(\al(r)\sigma(r))^q\Big)^{1/q}
\le \Big(\sum_{t\pc r\pq s}\vp(|r|)\Big)^{1/q}\\
&=&
\Big(\sum_{|t|<k\le |s|}\vp(k)\Big)^{1/q}\le \Big(\int_{|t|}^{|s|}
\vp(x)\d x\Big)^{1/q}=\bar d\big(|t|,|s|\big)^{1/q}\;.
\eeaa
Hence, if $\bar S_\eps=M_\eps\cup\set{u_1,\ldots,u_N}$ is defined as in (\ref{Deps0}), setting
\be
\label{D}
S_\eps:=\set{t\in T : |t|\in \bar S_{\eps^q}}
\ee
we obtain an order $2^{1/q}\eps$--net for $(T,d)$.

To proceed further we have to estimate $\# S_\eps$ suitably.
In view of $R(n)\le \rho(n)$ we get
\beaa
\# S_\eps &\le& \sum_{n\le \vp^{-1}(\eps^q)} \rho(n)+\sum_{k=1}^N \rho(u_k)
\le
\sum_{n\le \vp^{-1}(\eps^q)} \rho(n)+\sum_{k=1}^N \rho(\tilde u_k)\\
&=&
\sum_{n\le \vp^{-1}(\eps^q)} \rho(n)+\sum_{k=1}^N \rho(\Phi^{-1}(k\eps^q))\;.
\eeaa
Since $\rho$ is non--decreasing and $\rho\circ\Phi^{-1}$ non--increasing, this leads to
\beaa
\# S_\eps&\le& \int_0^{\vp^{-1}(\eps^q)+1}\rho(x)\d x + \rho(\Phi^{-1}(\eps^q))+
\int_1^{\Phi(\vp^{-1}(\eps^q))/\eps^q}
\rho(\Phi^{-1}(x\eps^q))\d x\\
&=&
\int_0^{\vp^{-1}(\eps^q)+1}\rho(x)\d x + \rho(\Phi^{-1}(\eps^q))+
\eps^{-q}\,\int_{\vp^{-1}(\eps^q)}^{\Phi^{-1}(\eps^q)}
 \rho(y)\vp(y)\d y\;.
\eeaa
Finally, we use $\tilde N(T,d,2^{1/q}\eps)\le \# S_\eps$ and replace $\eps^q$ by  $\eps^q/2$.
This proves the proposition.
\end{proof}

One can slightly simplify the bound for subsequent use as follows.

\begin{cor}
\label{cor2}
~
\bee
\item
\textbf{Convergent case:}\\
Suppose that
$
\int_1^\infty\rho(y)\,\vp(y)\d y<\infty.
$
Then it follows that
\be
\label{conv}
\tilde N(T,d,\eps)\le
\int_0^{\vp^{-1}(\eps^q/2)+1}\rho(x)\d x + \rho(\Phi^{-1}(\eps^q/2))
+2\,\eps^{-q}\,
\int_{\vp^{-1}(\eps^q/2)}^\infty
 \rho(y)\vp(y)\d y\;.
\ee
\item
\textbf{Divergent case:}\\
If
$
\int_1^\infty\rho(y)\,\vp(y)\d y=\infty\,,
$
then we get
\be
\label{div}
\tilde N(T,d,\eps)\le
\int_0^{\vp^{-1}(\eps^q/2)+1}\rho(x)\d x + \rho(\Phi^{-1}(\eps^q/2))+
2\,\eps^{-q}\,
\int_1^{\Phi^{-1}(\eps^q/2)}   \rho(y)\vp(y)\d y\;.
\ee
\eee
\end{cor}

Let us state and prove a first application of Proposition \ref{p7}
in the case of moderate trees, i.e.~those where the number of elements in the generations
increases at most polynomially.

\begin{prop}
\label{p8}
Let $T$ be a tree such that $R(n)\le c\,n^\H $ for some $\H  \ge 0$. Suppose, furthermore,
that
$$
   \al(t)\sigma(t) \le c\,|t|^{-\gamma/q}\,,\quad t\in T\,,
$$
for some $\gamma>1$. Then it follows
$$
\tilde N(T,d,\eps)\le c\,
\left\{
\begin{array}{lcl}
\eps^{-\frac{q\H }{\gamma-1}}&:& \gamma< \H +1\\
\eps^{-q}\log(1/\eps)&:& \gamma=\H +1\\
\eps^{-\frac{q(\H +1)}{\gamma}}&:& \gamma>\H +1.
\end{array}
\right.
$$
\end{prop}

\begin{proof}
First we note that in all three cases the behavior of the first and the second term
in (\ref{conv}) or (\ref{div}) is $\eps^{-\frac{q(\H +1)}{\gamma}}$ and
$\eps^{-\frac{q \H }{\gamma-1}}$, respectively. Only the third term behaves differently
in each of the three different cases.

Thus let us start with the
investigation of this third term  in the convergent case, i.e., if
$\gamma>\H +1$. We
use (\ref{conv}) and observe that the third term behaves as the first term, i.e., as
$$
    c\, \eps^{-q}\left[\vp^{-1}(\eps^q/2)\right]^{\H -\gamma+1}
    \le c \, \eps^{-\frac{q (\H+1) }{\gamma}}\;.
$$
Since here
  $\frac{\H }{\gamma-1}<\frac{\H +1}{\gamma}$,
the second term in (\ref{conv}) is of the smaller order and we obtain
$$
\tilde N(T,d,\eps) \le c\, \eps^{-\frac{q (\H+1) }{\gamma}}
$$
as asserted.

Next assume $\gamma=\H +1$. This is a kind of divergent case and the third term in (\ref{div})
is of order $\eps^{-q}\,\log(1/\eps)$,
while the first and the second term are of lower order $\eps^{-q}$, and we get
$$
\tilde N(T,d,\eps)\le c\, \eps^{-q}\log(1/\eps)
$$
as claimed above.

Finally, suppose
$\gamma<\H +1$. This is again a divergent case  and  the third term in (\ref{div}) behaves like
$$
    \eps^{-q}\,\left[ \Phi^{-1}(\eps^q/2)\right]^{\H -\gamma+1} \le c\, \eps^{-\frac{q\H}{\gamma-1}}\;,
$$
thus the second and the third term are of the same order.  Since for $\gamma<\H +1$ we have
  $\frac{\H }{\gamma-1}> \frac{\H +1}{\gamma}$,  the first term
that behaves like
$\eps^{-\frac{q (\H+1) }{\gamma}}$
is of smaller order. Thus it follows that
$$
\tilde N(T,d,\eps) \le c\, \eps^{-\frac{q\H}{\gamma-1}}
$$
which completes the proof.
\end{proof}

An application of Theorem \ref{t1} to the above estimates leads to the following.

\begin{thm}
\label{t2}
Suppose $1<q<\infty$ and let as before $p:=\min\set{2,q}$. Suppose that the tree $T$ satisfies
$R(n)\le c\,n^\H $ for a certain $\H \ge 0$ and that
$$
   \al(t)\sigma(t)  \le c\,|t|^{-\gamma/q}\,,\quad t\in T\,,
$$
for a certain $\gamma>1$. Then we may estimate the entropy numbers of the weighted summation operator
$V_{\al,\sigma}$ as follows:
$$
e_n(V_{\al,\sigma} : \ell_1(T)\mapsto \ell_q(T))\le c
\left\{
\begin{array}{lcl}
n^{-\frac{\gamma-1}{q\H }-\frac{1}{p'}} &:& \gamma<\H +1\\
n^{-\frac{1}{q}-\frac{1}{p'}}(\log n)^{1/q}&:& \gamma=\H +1\\
n^{-\frac{\gamma}{q(\H +1)}-\frac{1}{p'}} &:& \gamma>\H +1\;.
\end{array}
\right.
$$
\end{thm}

\begin{remark}
\rm
Note that $p=q$ for $1<q\le 2$. In particular, in that case
$$
%\label{uppert2}
e_n(V_{\al,\sigma} : \ell_1(T)\mapsto \ell_q(T))\le c
\left\{
\begin{array}{lcl}
n^{-\frac{\gamma-1}{q\H }- \frac{1}{q'}} &:& \gamma<\H +1\\
n^{-1}(\log n)^{1/q}&:& \gamma=\H +1\\
n^{-\frac{\gamma}{q(\H +1)}- \frac{1}{q'}} &:& \gamma>\H +1\;.
\end{array}
\right.
$$
\end{remark}
\bigskip

Our next objective is to investigate weighted summation operators on binary trees. Here
we have $\rho(x)=2^x$. Let us first suppose that the weights decay polynomially, i.e., we suppose
$$
  \al(t)\sigma(t)  \le c\,|t|^{-\gamma/q}\,,\quad t\in T\,,
$$
for some $\gamma>1$.
Of course, in order to estimate $\tilde N(T,d,\eps)$ we have to use the divergent case
of Corollary \ref{cor2}. Then we get
\beaa
   \log \int_0^{\vp^{-1}(\eps^q/2)+1}\rho(x)\d x&\approx&  \eps^{-q/\gamma}\quad\mbox{and}\\
   \log \rho(\Phi^{-1}(\eps^q/2))&\approx&  \eps^{-q/(\gamma-1)}\;.
\eeaa
Furthermore, as can be seen easily the logarithm of the third term in (\ref{div}) behaves
like $\eps^{-q/(\gamma-1)}$ as well.

Summing up, it follows that
$$
\log \tilde N(T,d,\eps)\le c\,\eps^{-q/(\gamma-1)}\;.
$$
Hence we see that the critical case appears if $q/(\gamma-1) = p'$ (recall that $p=\min\set{2,q}$), i.e.,
in the case
\beaa
\gamma=q       \quad &\mbox{if}& \quad 1<q\le 2\quad\mbox{and}\\
\gamma= q/2 +1 \quad &\mbox{if}& \quad 2\le q <\infty\;.
\eeaa
In the non-critical cases we get the following.

\begin{thm}
\label{t4}
Let $T$ be a binary tree and suppose that
$$
   \al(t)\sigma(t) \le c\,|t|^{-\gamma/q}\,,\quad t\in T\,,
$$
for a certain $\gamma>1$ with $\gamma\not= q/p'+1$. Then this implies

a) for $1<q\le 2$:

$$
e_n(V_{\al,\sigma}: \ell_1(T)\mapsto\ell_q(T))\le c\,
\left\{
\begin{array}{lcl}
n^{-\frac{1}{q'}}(\log n)^{1-\frac{\gamma}{q}}&:& \gamma>q\\
n^{-\frac{\gamma-1}{q}}&:& \gamma<q \;.
\end{array}
\right.
$$

b) for $2\le q<\infty$:

$$
e_n(V_{\al,\sigma}: \ell_1(T)\mapsto\ell_q(T))\le c\,
\left\{
\begin{array}{lcl}
n^{-\frac{1}{2}}(\log n)^{\frac{1}{2}-\frac{\gamma-1}{q}}&:& \gamma>\frac{q}{2}+1\\
n^{-\frac{\gamma-1}{q}}&:& \gamma<\frac{q}{2}+1\;.
\end{array}
\right.
$$

\end{thm}

\begin{remark}
\rm
For one--weight operators, i.e., if $\sigma(t)=1$, $t\in T$, and for $q=2$ the preceding result was
also proved in \cite{Lif10}. Moreover, it was shown there that the above estimates are sharp. But
the main result in \cite{Lif10} is the investigation of the critical case $\gamma=2$ if $q=2$. As mentioned
above the general
results for the entropy of the convex hull in \cite{CSt} lead only to
$$
e_n(V_{\al,\sigma}: \ell_1(T)\mapsto\ell_q(T))\le c\,n^{-\frac{\gamma-1}{q}}\,\log n
$$
in the critical case $\gamma= q/p'+1$.
\end{remark}
\bigskip

Let us shortly mention a third example. Again we take a binary tree $T$, but this time the weights
decrease exponentially, i.e., we assume
$$
    \al(t)\sigma(t)\le c\, 2^{-\frac{\gamma}{q}|t|}\;,\quad t\in T\,,
$$
for some $\gamma>0$. Hence we have $\vp(x)=2^{-\gamma x}$ and
$$
\vp^{-1}(\eps^q)\sim\Phi^{-1}(\eps^q)\sim \frac{q}{\gamma}\log_2(1/\eps)\;.
$$
Thus all terms in (\ref{conv}) and (\ref{div}) are of the same order $\eps^{-q/\gamma}$
and under these assumptions
$$
\tilde N(T,d,\eps)\le c\,\eps^{-q/\gamma}\;.
$$
Thus, it follows
\be
\label{expcase}
e_n(V_{\al,\sigma}:\ell_1(T)\mapsto \ell_q(T))\le c\, n^{-\frac{\gamma}{q}-\frac{1}{p'}}
\ee
in that case.
In completely different probabilistic language, this example was studied in \cite{AurLif}.

\section{Examples of Lower Entropy Estimates}
\label{s:lowermod}
\setcounter{equation}{0}

In Section \ref{s:uppermod} we proved upper estimates for $N(T,d,\eps)$ under
certain growth assumptions for the weights and for $R(n)$, the number of elements in the $n$--th
generation of $T$. The aim of this section is to prove  in similar way lower
estimates for $N(T,d,\eps)$ or $e_n(V_{\al,\sigma})$, respectively, assuming lower growth
estimates.
Thus we investigate weights satisfying
\be
\label{lowweight}
(\alpha(t)\sigma(t))^q\ge \vp(|t|)\,,\quad t\in T\,,
\ee
for a function $\vp$ as in Section \ref{s:uppermod}
and, furthermore, we assume
\be
\label{lowR}
R(n)\ge \rho(n)\,,\quad n\in\N_0\,,
\ee
where $\rho$ is as before non--increasing and continuous with $\rho(0)=1$.
\medskip

Under these assumptions we get the following.
\begin{prop}
\label{p9a}
Assume $(\ref{lowweight})$ and $(\ref{lowR})$. Then we have
$$
N(T,d,\eps/2)\ge \int_1^{\vp^{-1}(\eps^q)-1}\rho(x)\d x\;.
$$
\end{prop}
\begin{proof}
Fix $\eps>0$ and set
$$
T_\eps:=\set{t\in T : 0\le |t|\le \vp^{-1}(\eps^q)}\;.
$$
Given $s\in T$, $s\not=\mathbf 0$,  let $s'$ be the parent element of $s$, i.e., $s$ is
an offspring of $s'$. Then (\ref{lowweight}) implies
$$
d(s',s)=\al(s)\sigma(s)\ge\vp(|s|)^{1/q}\ge \eps
$$
provided that $s\in T_\eps$. Let now $t,s\in T_\eps$ with $t\not= s$. If $t\pc s$, then
$t\pq s'\pc s$, hence $d(t,s)\ge d(s',s)\ge \eps$. Otherwise, i.e.~if $t$ and $s$ are
incomparable, by the same argument we get
$$
d(t,s)\ge d(t\wedge s,s)\ge \eps
$$
as well. Consequently, $T_\eps$ is an $\eps$--separated subset of $T$ which implies
$$
N(T,d,\eps/2)\ge \# T_\eps\;.
$$
Thus, in order to complete the proof it suffices to estimate $\# T_\eps$ suitably. Here we
use (\ref{lowR}) and obtain
$$
\# T_\eps = \sum_{0\le n\le \vp^{-1}(\eps^q)}  R(n)\ge
\sum_{0\le n\le \vp^{-1}(\eps^q)}  \rho(n)
\ge
\int_1^{\vp^{-1}(\eps^q)-1}\rho(x)\d x\;,
$$
as asserted.
\end{proof}

\begin{cor}
\label{c2}
Suppose that
$$
\int_1^{\vp^{-1}(\eps^q)-1}\rho(x)\d x\ge c\, \eps^{-a}\abs{\log\eps}^b
$$
for certain $a>0$ and $b\ge 0$. Then this implies
\be
\label{enV}
e_n(V_{\al,\sigma} : \ell_1(T)\mapsto \ell_q(T))\ge \tilde c\, n^{-1/a-1/q'}(\log n)^{b/a}\;.
\ee
\end{cor}

\begin{proof}
Using Proposition \ref{p9a} the assumption leads to
$$
N(T,d,\eps)\ge c'\eps^{-a}\abs{\log\eps}^b\;.
$$
Consequently, Theorem \ref{t3} applies and proves (\ref{enV}).
\end{proof}

Let us apply the preceding corollary for concrete functions $\vp$ and $\rho$. We start with
the investigation of moderate trees and polynomial weights, i.e., $\rho$ is of polynomial growth and
$\vp(x)\ge c\,x^{-\gamma}$ for a certain $\gamma>1$. Here we get

\begin{prop}
\label{p9}
Suppose that $T$ is a tree with $R(n)\ge c\,n^{\H }$ for some $\H \ge 0$. Furthermore
assume
$$
\al(t)  \sigma(t)  \ge c\,|t|^{-\gamma/q}\,,\quad t\in T\,,
$$
for some $\gamma>1$. Then it follows that
\be\label{lowerp9}
e_n(V_{\al,\sigma} : \ell_1(T)\mapsto \ell_q(T))\ge \tilde c\,n^{-\frac{\gamma}{q(\H +1)}-\frac{1}{q'}}\;.
\ee
\end{prop}

\begin{proof}
This follows directly from Corollary \ref{c2} by evaluating the integral.
\end{proof}

\begin{remark}
\rm
Suppose $1<q\le 2$.
Then the  preceding proposition shows that the estimates in Theorem \ref{t2} are sharp provided that
$\gamma>\H +1$. We will see later on that this is no longer always true if $1<\gamma\le \H+1$.
\end{remark}
\bigskip

Another application of Proposition \ref{p9a} leading to sharp lower estimates is as follows.

\begin{prop}
\label{p9b}
Let $T$ be a binary tree and suppose that
$$
\al(t)\sigma(t)\ge c\,2^{-\frac{\gamma}{q}|t|}\,,\quad t\in T\;,
$$
for some $\gamma>0$.
Then this implies
$$
e_n(V_{\al,\sigma} :\ell_1(T)\mapsto \ell_q(T))\ge c\,n^{-\frac{\gamma}{q}-\frac{1}{q'}}\;.
$$
\end{prop}

\begin{proof}
Again this is a direct consequence of Corollary \ref{c2} and the fact that
$$
\vp^{-1}(\eps^q) = \frac{q}{\gamma}\log_2(1/\eps)  +\frac{\log_2 c}{\gamma}\;.
$$
Recall that $\rho$ may be chosen as $\rho(x)=2^x$ in that case.
\end{proof}

\begin{remark}
\rm
Combining the preceding proposition with (\ref{expcase}) gives the following:
Let $T$ be a binary tree and suppose $1<q\le 2$. If
$$
\al(t)\sigma(t)\approx 2^{-\frac{\gamma}{q}|t|}\,,\quad t\in T\;,
$$
then
$$
e_n(V_{\al,\sigma} :\ell_1(T)\mapsto \ell_q(T))\approx n^{-\frac{\gamma}{q}-\frac{1}{q'}}\;.
$$
\end{remark}

As we said above, Proposition \ref{p9} does not always lead to sharp lower estimates,
 even if $1<q\le 2$. The reason is that here the structure of the underlying
tree plays a role. We assume now
that
$\xi(t)\ge 1$ for each $t\in T$. In different words, we suppose that each element in $T$ has at least
one offspring. Furthermore we restrict ourselves to one--weight operators defined as follows.
We write $V_\al$ instead of
$V_{\al,\sigma}$ provided that $\sigma\equiv 1$, i.e., $V_\al$ denotes the one--weight operator
acting as
\be  \label{Val}
(V_\al x)(t):=\al(t)\sum_{s\succeq t} x(s)\,,\quad t\in T\;.
\ee
\bigskip

In the case of a one--weight operator condition (\ref{lowweight}) reads now as
\be
\label{low2}
\alpha(t)\ge \vp(|t|)^{1/q}\,,\quad t\in T\;.
\ee
To proceed further we have to exclude functions $\vp$ decreasing too fast. Thus we assume
that there is a constant $\kappa\ge 1$ such that
\be
\label{Delta2}
\vp(x)\le \kappa\,\vp(2 x)\,,\quad x\ge x_0\;.
\ee
Let $\Phi$ be defined as in (\ref{Phi}). For later use we mention that (\ref{Delta2}) implies
$$
%\label{low3}
\frac{\Phi(x)}{\vp(x)}\ge x\,\frac{\vp(2 x)}{\vp(x)}\ge \kappa^{-1}\,x\,,\quad x\ge x_0\;,
$$
hence
\be
\label{low3}
\frac{\Phi(\vp^{-1}(y))}{y}\ge \kappa^{-1}\,\vp^{-1}(y)\to\infty
\quad\mbox{as}\quad y\to 0\;.
\ee
Under these assumptions we get the following general lower estimate.

\begin{prop}
\label{p9c}
Let $T$ be a tree with $\xi(t)\ge 1$ for each $t\in T$ and suppose $(\ref{low2})$ as well as
$(\ref{Delta2})$. Then there is an $\eps_0>0$ such that
\be
\label{NTd}
N(T,d,\eps/2)\ge 4^{-1}\eps^{-q}\int_{\vp^{-1}(\eps^q)}^{\Phi^{-1}(8\eps^q)}\rho(y)\vp(y)\d y
\ee
whenever
$0<\eps<\eps_0$.
\end{prop}

\begin{proof}
First note that for one--weight operators the metric $d$ reduces to
$$
d(t,s)= \Big(\sum_{t\pc r\pq s}\al(r)^q\Big)^{1/q}
$$
whenever $t\pq s$. By (\ref{low2}) this implies
\be
\label{distbar}
d(t,s)\ge \Big(\sum_{k=|t|+1}^{|s|}\vp(k)\Big)^{1/q}\ge \bar d\big(|t|+1,|s|\big)^{1/q}
\ee
with $\bar d $ defined in (\ref{dbar}).

We construct now positive real numbers
$\tilde u_1>\tilde u_2>\cdots>\tilde u_N$ as in (\ref{uktilde}) but this time directly
with $\eps^q$ instead of $\eps$, i.e.,
$$
\tilde u_k := \Phi^{-1}(k\,\eps^q)\,,\quad 1\le k\le N\,,
$$
where $N$ satisfies
$$
N\le \frac{\Phi(\vp^{-1}(\eps^q))}{\eps^q}< N+1\;.
$$
Next set
$$
v_k :=[\tilde u_{3k}]+1\,,\quad 1\le k\le m\,,
$$
where
$$
%\label{mest}
m=\left[\frac N 3\right]\ge \frac{\Phi(\vp^{-1}(\eps^q))}{3\eps^q} -2\;.
$$
Using (\ref{low3}) this implies
\be
\label{low4}
m \ge
\frac{\Phi(\vp^{-1}(\eps^q))}{4\eps^q}
\ee
provided that $\eps<\eps_0$ for a certain $\eps_0$ depending on $\vp$.

Since $\tilde u_k -\tilde u_{k-1}>1$ we get
$$
[\tilde u_{3k-2},\tilde u_{3k-3}]\subseteq [v_k+1, v_{k-1}]\,,
$$
hence
\be
\label{vest}
\bar d(v_k+1,v_{k-1})\ge \eps^q\;.
\ee

Let us construct now an $\eps$--separated subset $S_\eps\subseteq T$ as follows. For $1\le k\le m$
set
$$
T_k :=\set{t \in T : |t|=v_k}
$$
and given $t\in T_k$ with $2\le k\le m$ we choose exactly one $s_{k-1}(t)\in T_{k-1}$
satisfying $s_{k-1}(t) \succ t$.
Those $s_{k-1}(t)$ exist because we assumed $\xi(t)\ge 1$ for all $t\in T$. Finally, define
$S_\eps$ by
$$
S_\eps:=\bigcup_{k=2}^{m}\set{s_{k-1}(t) : t\in T_k}\;.
$$
Because of (\ref{distbar}) and (\ref{vest}) the points in $S_\eps$ are $\eps$--separated and
since $\tilde u_{3 k}\le v_k$ the properties of $\rho$ yield
$$
\# S_\eps = \sum_{k=2}^{m}\# T_k =\sum_{k=2}^{m} R(v_k)\ge \sum_{k=2}^{m}\rho(v_k)
\ge \sum_{k=2}^{m} \rho(\tilde u_{3 k}) =\sum_{k=2}^{m}\rho(\Phi^{-1}(3 k\eps^q))\;.
$$
Clearly this implies
$$
N(T,d,\eps/2)\ge\sum_{k=2}^{m}\rho(\Phi^{-1}(3 k\eps^q))\;.
$$
Observe that $\rho\circ \Phi^{-1}$ is decreasing and recall (\ref{low4}). Then we get
\beaa
N(T,d,\eps/2)&\ge& \int_2^{\frac{\Phi(\vp^{-1}(\eps^q))}{4\eps^q}}\rho(\Phi^{-1}(3 x \eps^q))\d x
\ge \int_2^{\frac{\Phi(\vp^{-1}(\eps^q))}{4\eps^q}}\rho(\Phi^{-1}(4 x \eps^q))\d x\\
&=& 4^{-1}\,\eps^{-q}\int_{\vp^{-1}(\eps^q)}^{\Phi^{-1}(8\eps^q)}\rho(y)\vp(y)\d y
\eeaa
as asserted.
\end{proof}
\begin{remark}
\rm
Unfortunately, we do not know whether or not an estimate similar to (\ref{NTd})
remains valid in the case of two weights $\al$ and $\sigma$ satisfying
(\ref{lowweight}). The crucial
point is that in this case estimate (\ref{distbar}) is no longer valid. For example, take
$\vp(x)=x^{-\gamma}$ for some
$\gamma>1$
and choose the weights as
$\al(t)=2^{|t|/q}$ and $\sigma(t)= |t|^{-\gamma/q}2^{-|t|/q}$ to see that (\ref{distbar}) is not
satisfied in general.
\end{remark}
\bigskip

A first application is for moderate trees with polynomial decay of the weight $\al$. It shows that
the estimates in Theorem \ref{t2} are also sharp (at least for one--weight operators and $1<q\le 2$)
for $1<\gamma\le \H+1$, provided we have the additional
assumption $\xi(t)\ge 1$ for $t\in T$.

\begin{prop}
\label{p91}
Let $T$ be a tree with $\xi(t)\ge 1$ for $t\in T$ such that $R(n)\ge c\,n^\H $ . Given $\gamma>1$ let
$\al(t)\ge c\,|t|^{-\gamma/q}$.
Then, if $\gamma<\H  +1$, it follows that
\be
\label{p91a}
N(T,d,\eps)\ge c\, \eps^{-\frac{q\H}{\gamma-1}}
\ee
Similarly, if $\gamma=\H+1$, then
\be
\label{p91b}
N(T,d,\eps)\ge c\,\eps^{-q}\log(1/\eps)\;.
\ee
For the operator $V_\al$ we have
\[
e_n(V_{\al} : \ell_1(T)\mapsto \ell_q(T))\ge c
\left\{
\begin{array}{lcl}
n^{-\frac{\gamma-1}{q\H }-\frac{1}{q'}} &:& \gamma<\H +1\\
n^{-1}(\log n)^{1/q}&:& \gamma=\H +1.
\end{array}
\right.
\]
\end{prop}

\begin{proof}
Of course, $\vp(x)=c^q x^{-\gamma}$ satisfies condition (\ref{Delta2}).
Thus Proposition \ref{p9c} applies and the lower estimates in (\ref{p91a}) and (\ref{p91b})
are direct consequences of $\vp^{-1}(\eps^q)\approx \eps^{-q/\gamma}$ as well as of $\Phi^{-1}(\eps^q)\approx
\eps^{-q/(\gamma-1)}$.
The estimates for $e_n(V_{\al})$ now follow from Theorem $\ref{t3}$.
\end{proof}

Another application of Proposition \ref{p9c} is for binary trees and polynomial decay of $\al$.

\begin{prop}
\label{p92}
Let $T$ be a binary tree and suppose $\al(t)\ge c\,|t|^{-\gamma/q}$ for some $\gamma>1$. Then this yields
$$
\log N(T,d,\eps)\ge c\, \eps^{-q/(\gamma-1)}\;.
$$
\end{prop}

\begin{proof}
Again $\vp$ satisfies (\ref{Delta2}), hence Proposition \ref{p9c} applies as well
and the assertion easily follows by
$$
\log \int_{\vp^{-1}(\eps^q)}^{\Phi^{-1}(8\eps^q)}\rho(y)\vp(y)\d y
\approx \Phi^{-1}(8\eps^q)
\approx \eps^{-q/(\gamma-1)}\;.
$$
\end{proof}

Combining Proposition \ref{p92} with Theorem \ref{t3a} leads to the following.
\begin{prop}
\label{p93}
Let $T$ be a binary tree and suppose that $\al(t)\ge c\,|t|^{-\gamma/q}$ for a certain $\gamma>1$.
Then this implies
\[
e_n(V_{\al} : \ell_1(T)\mapsto \ell_q(T))\ge c
\left\{
\begin{array}{lcl}
n^{-1/q'}(\log n)^{1-\,\gamma/q} &:& \gamma>q\\
n^{-(\gamma-1)/q}&:& \gamma\le q.
\end{array}
\right.
\]
\end{prop}

\section{Biased Trees}
\setcounter{equation}{0}
 \label{s:bias}

We will now test the sharpness of our bounds on an interesting class of trees whose branches, opposite
to the case of Proposition \ref{p91}, die out quickly. Let $\H \ge 1$. We define a biased tree of order $\H $
as follows. Take a binary tree, draw it on the plane so that it grows from the bottom to the top,
and for any level $n\ge 0$ keep only the $R(n)$ rightmost nodes where
\[
   R(n):=
   \begin{cases}
     2^n, & n\le 2\H,\\
     n^\H ,& n> 2\H.
   \end{cases}
\]
The set of nodes we have kept is a tree since
\[
   R(n+1)\le 2R(n)\,,\qquad n\ge 0\;.
\]
We call this tree a {\it biased tree} (because it is really biased to the right)
of order $\H $ and denote it by $T_\H $. Since the size of its $n$-th level for large $n$ is $n^\H $,
the biased tree satisfies both the upper and lower size bounds
\be \label{Rbounds}
 c\, n^\H \le R(n) \le C\, n^\H,
\ee
as in Theorem \ref{t2} and in Proposition \ref{p9}, respectively. At the same time
the nodes situated on large levels die out pretty quickly, which enables more efficient covering
constructions than in the general case.

On $T_\H $ we will consider the usual one--weight operator $V_\al$ defined in (\ref{Val}).
Recall that $\tilde N(T_\H,d,\eps)$ stands for the order covering numbers of the tree $T_\H$
defined in (\ref{order}). Our main result for biased trees is the following.

\begin{prop} \label{main_bias}
Let $V_\al$ be the one--weight operator on $T_\H$  with the weight $\al(t)=|t|^{-\ga/q}$, $\ga>1$,
and let $d$ be the metric on $T_\H$ corresponding to this weight.
For the related order covering numbers we have
\begin{equation} \label{Nbias}
    c\, \eps^{-\frac{q(\H +1)}{\ga}} \le \tilde N(T_\H,d,\eps) \le C\, \eps^{-\frac{q(\H +1)}{\ga}}.
\end{equation}
For the entropy numbers of $V_\al$ we have
\begin{equation} \label{Vbias}
    c\,  n^{-\frac{\ga}{q(\H +1)}-\frac{1}{q'}}\le  e_n(V_\al) \le C\,  n^{-\frac{\ga}{q(\H +1)}-\frac{1}{q'}}.
\end{equation}
\end{prop}

This shows that the lower bound (\ref{lowerp9}) of Section \ref{s:lowermod}
can not be improved in general, unless we make some extra assumptions
about the tree -- as in Proposition \ref{p91}.

We also see from this bound that the upper estimates  for order covering numbers obtained
in Proposition \ref{p8} and those for entropy numbers obtained in Theorem \ref{t2} are not sharp for
certain trees in the convergent case $\ga < \H+1$ and in the intermediate case $\ga=\H+1$, while
the results of Sections \ref{s:uppermod} and \ref{s:lowermod} show that in the divergent case ($\ga> \H+1$)
the estimate for the entropy numbers is sharp for {\it any} tree satisfying (\ref{Rbounds}).
\medskip

Now we start proving Proposition \ref{main_bias}.

\begin{proof}
The construction will be based on the same set of levels as in (\ref{Deps0}) but we specify it for
our situation.
Let
\begin{equation} \label{A}
   \Phi(y)= \int_y^\infty x^{-\ga}\d x= c\, y^{-(\ga-1)}.
\end{equation}
We will use the following elementary property. For any positive integers $n< m$ it is true that
\begin{equation} \label{nA}
   \sum_{k=n+1}^m   k^{-\ga} \le \int_n^m x^{-\ga} \d x = \Phi(n)-\Phi(m).
\end{equation}
Given $\eps\in (0,1)$, let $J=[\eps^{-q/\ga}]$ and
choose a decreasing sequence of integers $(n_j)_{1\le j\le J}$, by
$$
%\label{nj}
  n_j := \inf \{n\in \N:\ \Phi(n)\le j\eps^q \}.
$$
We also let $n_0:=+\infty$ for uniformity of further writing.
By (\ref{A}) we have
$$
%\label{njbound}
    n_j \le C (j \eps^q)^{-\frac{1}{\ga-1}}.
$$
In particular, we have
\be \label{nJbound}
   n_J \le C \, \eps^{-\frac{q}{\ga}}.
\ee
Now we define our order net as follows:
$S_\eps:=S_\eps^1\cup S_\eps^2$, and $S_\eps^1:=\{s: |s|<n_J\}$, whereas
\[
  S_\eps^2:=\bigcup_{j=1}^J S_{\eps,j},
\]
and $S_{\eps,j}$ consists of the first $\nu_j:= \min\{c_*\eps^{-\frac{q\H }{\ga}}, R(n_j)\}$ rightmost
nodes of the level $n_j$. The large constant $c_*$ will be specified later. Recall that in the construction
we used to prove Proposition \ref{p7} we included in the net the entire levels, see (\ref{D}).
Due to the structure of the biased tree, only a small part of the level suffices, thus the net is more efficient.

The size of the net is bounded by
\begin{eqnarray*}
\#\, S_\eps &\le& \sum_{n=1}^{n_J} R(n) +\sum_{j=1}^J \nu_j
\le C \sum_{n=1}^{n_J} n^\H  + J\cdot c_*\eps^{-\frac{q\H}{\ga}}
\le C n_J^{\H +1} + J\cdot c_*\eps^{-\frac{q\H }{\ga}}
\\
&\le& C \eps^{-\frac{q(\H +1)}{\ga}} \qquad \textrm{by}\ (\ref{nJbound})\ \textrm{and by the definition
of}\  J.
\end{eqnarray*}
In order to evaluate the precision of the net, we will use the following structural property
of the biased tree.

\begin{lem} \label{biaprop}
Let $j\le J$ and let $s\in T$ be such that $|s|\ge n_j$\, . Then there exists a $t\in S_{\eps,j+1}$
such that $t\prec s$.
\end{lem}

\begin{proof}
First of all, notice that it is enough to consider the case $|s|=n_j$. Indeed,
for any $s$ with $|s|\ge n_j$ we find $s'$ satisfying $|s'|=n_j$ and $s'\preceq s$. Once the lemma
is proved for $s'$, we find an appropriate $t\in S_{\eps,j+1}$ for $s'$ and conclude from $t\prec s'\preceq s$ that
$t\prec s$.

 So let us assume that $|s|=n_j$. Now look at  $\nu_{j+1}= \min\{c_*\eps^{-\frac{q\H }{\ga}}, R(n_{j+1})\}$.
 If  $\nu_{j+1}= R(n_{j+1})$, this means that $S_{\eps,j+1}$ coincides with the entire $n_{j+1}$-th level of
 $T_\H $. Then of course there exists $t\in S_{\eps,j+1}$ such that $t\prec s$.

 On the other hand, if $\nu_{j+1}= c_*\eps^{-\frac{q\H }{\ga}}$, then our statement reduces to the numerical
 inequality
 \be \label{compare}
     c_*\eps^{-\frac{q\H }{\ga}} \cdot 2^{n_j-n_{j+1}} \ge \tilde C\, n_{j}^\H .
 \ee
Here the left hand side is the total number of offsprings of elements in $S_{\eps,j+1}$ counted on
the $n_j$-th level of zhe initial binary tree and the right hand side is an upper bound for the size $R(n_j)$ of
$n_j$-th level in $T_\H $.

It follows from the definition of $n_j$ that $n_j \sim c (j \eps^q)^{-\frac{1}{\ga-1}}$, hence
\[
n_{j+1}-n_j\ge c \, \eps^{-\frac{q}{\ga-1}} j^{-(1+\frac{1}{\ga-1})}
\ge c \, \eps^{-\frac{q}{\ga-1}} (n_j^{\ga-1}\eps^q)^{1+\frac{1}{\ga-1}}
:= c_1 n_j^\ga \eps^q.
\]
If $c_*$ is large enough, then for any $x\ge 0$ we have
$2^{c_1 x}\ge  c_*^{-1} \tilde C x^{\frac \H \ga}$. By letting here $x=n_j^\ga \eps^q$ we obtain
\[
2^{c_1 n_j^\ga \eps^q} \ge \tilde C c_*^{-1} n_j^\H  \eps^{\frac {q\H }\ga},
\]
hence,
\[
c_*\eps^{-\frac{q\H }{\ga}} \cdot 2^{n_j-n_{j+1}}
\ge
c_*\eps^{-\frac{q\H }{\ga}} \cdot 2^{c_1 n_j^\ga \eps^2}
\ge
c_*\eps^{-\frac{q\H }{\ga}} \cdot \tilde C c_*^{-1} n_j^\H  \eps^{\frac {q\H }\ga}
=
\tilde C n_{j}^\H ,
\]
and (\ref{compare}) follows.
\end{proof}

Now the precision of the net is easy to establish.
Recall that if $t\prec s$ then
\be
d(t,s)^q
=  \sum_{t\prec\ r \preceq s} \al(r)^q
 \label{altals}
= \sum_{|t|<k \le |s|} k^{-\ga}.
\ee

Next, if $s\not \in S_{\eps}$, only the two following cases are possible.

1) $|s|>n_1$. Apply Lemma \ref{biaprop} with $j=1$. We find $t\in S_{\eps,2}$ such that $t\prec s$. Then
by (\ref{altals}) and (\ref{nA}) we have
\[
d(t,s)^q
\le \sum_{k=n_2+1}^\infty k^{-\ga}  \le \Phi(n_2) \le  2 \eps^q.
\]

2) $n_j<|s|<n_{j-1}$ for some $2\le j\le J$. Apply Lemma \ref{biaprop} with $j$.
We find a $t\in S_{\eps,j+1}$ such that $t\prec s$. Then
by (\ref{altals})  and (\ref{nA}) we have
\[
d(t,s)^q
\le \sum_{k=n_{j+1}+1}^{n_{j-1}-1} k^{-\ga}   \le
\Phi(n_{j+1})-\Phi(n_{j-1}-1) \le
2 \eps^q.
\]

Therefore for any $s\in T_\H $ we have a $t\in S_\eps$ such that
$t\preceq s$ and $d(t,s) \le 2^{1/q}\, \eps$.
Taking into account the bound for $\#\, S_\eps$, we see that
$\tilde N(T_\H, d, \eps) \le C \eps^{-\frac{q(\H +1)}{\ga}}$.

For the lower bound, take any distinct $s,t\in S_\eps^1$, that is $|s|< n_J$, $|t|< n_J$.
Then
\[
d(s,t) \ge \max\{\al(s),\al(t)\} \ge (n_J)^{-\ga/q}
\ge c \eps,
\]
while the number of points we consider is bounded from below by
\[
 \#\, S_\eps^1 = \sum_{n=1}^{n_J-1} R(n)
 \ge c \, \sum_{n=1}^{n_J-1} n^\H
 \ge c \, (n_J-1)^{\H +1}
 \ge c \, \eps^{-\frac{q(\H +1)}{\ga}}.
\]
It follows that
\[
 \tilde N(T_\H, d,\eps) \ge  N(T_\H, d,\eps) \ge c \, \eps^{-\frac{q(\H +1)}{\ga}},
\]
as required in (\ref{Nbias}).
For the entropy numbers,  the upper bound in (\ref{Vbias}) follows from the upper bound in (\ref{Nbias})
via Theorem \ref{t1},
while the lower bound in (\ref{Vbias}) was proved in the more general context of
Proposition \ref{p9}, see (\ref{lowerp9}).
\end{proof}

\section{A Probabilistic Application}
\label{s:probab}
\setcounter{equation}{0}
Due to the well known relations between the entropy of operators on Hilbert spaces
and small deviation probabilities
of Gaussian random functions, our results have immediate probabilistic consequences.
Thus regard $V_{\al,\sigma}$ as operator from $\ell_1(T)$ into $\ell_2(T)$. Its
dual $V_{\al,\sigma}^*$ maps $\ell_2(T)$ into $\ell_\infty(T)$, hence it generates a Gaussian
 random function $X=(X_t)_{t\in T}$ by
$$
  X_t := \sum_{r\in T}\xi_r\, (V_{\al,\sigma}^* \delta_r)(t)=\sigma(t)\sum_{r\pq t}\al(r)\,\xi_r
$$
where $\{\xi_r,r\in T\}$ is a family of independent $\mathcal N(0,1)$--distributed random variables.
The covariance structure of $X$ is given by
$$
\E X_t X_s= \sigma(t)\sigma(s)\sum_{r\pq t\wedge s}\al(r)^2\,,\quad t,s\in T\;.
$$
Such summation schemes on trees are extensively studied and applied, see e.g.~the
literature on Derrida random energy model
\cite{BoK} or displacements in random branching walks \cite{Pem}, to mention just a few.

As one consequence of our results we get the following.
\begin{prop}
\label{sd}
Suppose $N(T,d,\eps)\approx\eps^{-a}\abs{\log\eps}^b$, for some $a>0,b\ge 0$.  Then
this implies that
$$
-\log\pr{\sup_{t\in T}\abs{X_t}<\eps}\approx \eps^{-a}\abs{\log\eps}^b\;.
$$
%The corresponding one--sided versions of this implication are valid as well.
\end{prop}
\begin{proof}
An application of Theorem \ref{t3} implies
$$
e_n(V_{\al,\sigma} : \ell_1(T)\mapsto\ell_2(T))\approx
n^{-1/a\,-1/2}(\log n)^{b/a}\;.
$$
Next, duality results for entropy numbers (cf.~\cite{TJ}) lead to
$$
e_n(V_{\al,\sigma}^* : \ell_2(T)\mapsto\ell_\infty(T))\approx
n^{-1/a\,-1/2}(\log n)^{b/a}
$$
as well.
Recall that $V_{\al,\sigma}^*$ generates $X$, hence we may apply
the classical Kuelbs--Li result (see \cite{KuLi} or \cite{LiLin}) and obtain
$$
-\log\pr{\sup_{t\in T}\abs{X_t}<\eps}\approx \eps^{-a}\abs{\log\eps}^b
$$
as asserted. %The one--sided implications are proved along the same lines.
\end{proof}
\begin{remark}
\rm
By the same methods one gets that $N(T,d,\eps)\le c\,\eps^{-a}\abs{\log\eps}^b$ yields
$$
-\log\pr{\sup_{t\in T}\abs{X_t}<\eps}\le c\, \eps^{-a}\abs{\log\eps}^b\;.
$$
Surprisingly, this looks exactly as a special case of a general small deviation result
due to M. Talagrand (cf.~\cite{Ta} or \cite{Led}). Yet the main difference is that in the cited result
one uses the covering numbers w.r.t.~the so--called Dudley distance $d_X$ while our results are based
on these numbers w.r.t.~the
metric $d$
defined in (\ref{metric}). This suggests that there is maybe some relation between $d_X$ and $d$.
Even if this is the case, it is at least not obvious.
\end{remark}

\section{Concluding Remarks and Open Problems}
\label{s:open}
\setcounter{equation}{0}
We must say that the study of summation operators on trees we merely initiated here is far from being
complete. For example, recall that many of our estimates are proven to be sharp only in the range
of the parameter $q\le 2$ while there are gaps for $q>2$. It would also be quite natural to consider the
operators acting from $\ell_p(T)$ into $\ell_q(T)$ with general $p,q\in[1,\infty]$.
In both cases the reason of difficulties is
that the technique of convex hulls that we refer to in Section \ref{s:upper1} is not appropriate anymore and
other tools are needed.

In this context let us mention the following related open question: Given
$1<p,q<\infty$ and a tree $T$. For which weights $\al$ and $\sigma$ is $V_{\al,\sigma}$ a bounded operator from $\ell_p(T)$
into $\ell_q(T)$ ? To our knowledge this is even unknown if $T$ is a binary tree. Let us shortly recall
the answer to this question in the case $T=\N_0$ (cf.~\cite{CL}
where it was derived from the classical Maz'ja--Rosin Theorem for weighted integration operators).
We formulate it only in the case $1\le p\le q\le \infty$ although the answer is known
for all $p,q\in[1,\infty]$.
\begin{prop}
\label{MR}
If $1\le p\le q\le \infty$, then $V_{\al,\sigma}$ is bounded from $\ell_p(\N_0)$ into $\ell_q(\N_0)$
if and only if
$$
\sup_{v\in \N_0} \norm{\al\on_{[0,v]}}_q\norm{\sigma\on_{[v,\infty)}}_{p'}<\infty\;.
$$
%In the case $1\le q<p\le \infty$ the answer is also known but more complicated to formulate.
\end{prop}
\medskip

We shortly mention two other problems related to the presented topic.

\noindent
(1) Throughout the paper we always assumed $\sigma$ to be non--increasing. This property was used
at several places. For example, it played an important role in the proofs of Proposition \ref{p1} and
\ref{p6}, respectively. If $\sigma$ is not necessarily non--increasing, then the  distance $d$ has surely to
be modified as
$$
\hat d(t,s):= \max_{t\pc v\pq s}\norm{\al\on_{(t,v]}}_q\norm{\sigma\on_{[v,s]}}_\infty
$$
whenever $t\pq s$.
Unfortunately, then, in general,  $\hat d$ can no longer be extended to a metric on $T$. Nevertheless we
believe that some covering properties of $T$ w.r.t.~$\hat d$ are tightly connected with compactness
properties of $V_{\al,\sigma}$. At least this is suggested by the known results for compactness
and approximation properties of weighted
integration operators as proved in \cite{EEH1} or \cite{LifLi}.
\medskip

\noindent
(2) A challenging problem is the critical case as treated in Theorem \ref{t1}.
 Some partial results related to the critical case are known. For example, in \cite{Lif10}
the problem is solved for the binary tree provided that $q=2$ and $\sigma(t)\equiv 1$. Other results
in the critical case we are aware of are based on \cite{CE} and will be handled in a separate publication.

\bigskip

\noindent
\textbf{Acknowledgement:}
\ The research was supported by the RFBR-DFG grant 09-01-91331 "Geometry and asymptotics of random structures".
The work of the first  named author was also supported by RFBR grants 09-01-12180-ofi\_m and
10-01-00154a.
%% NSh-4472.2010.1 "Leading scientific schools of Russia"

\bibliographystyle{amsplain}

\begin{thebibliography}{10}



{\baselineskip=12pt

\bibitem{AurLif}
Aurzada, F. and Lifshits, M. A.,
\textit {Small deviation probability via chaining.}
Stoch. Proc. Appl.~{\bf 118} (2008), 2344--2368.

\bibitem{BoK}
Bovier, A. and Kurkova, I.,
\textit {Derrida's generalized random energy models 1: models with
finitely many hierarchies.}
Ann. Inst. H. Poincar\'e. Probab. Stat.~{\bf 40} (2004), 439--480.
\bibitem{CE}
Carl, B. and Edmunds, D. E.,
\textit{Gelfand numbers and metric entropy of convex hulls in Hilbert spaces.}
Stud. Math. \textbf{159} (2003), 391-402.

\bibitem{CKP}
Carl, B., Kyrezi, I. and Pajor, A.,
\textit{Metric entropy of convex hulls in Banach spaces.}
J. London Math.~Soc.~{\bf 60} (1999), 871--896.

\bibitem{CS}
Carl, B. and Stephani, I.,
 \textit{Entropy, Compactness and Approximation of Operators.}
 Cambridge Univ. Press, Cambridge, 1990.

 \bibitem{CL}
Creutzig, J. and Linde, W.,
\textit{Entropy numbers of certain summation operators.}
Georgian Math. J. \textbf{8} (2001), 245-274.

\bibitem{CSt}
Creutzig, J. and Steinwart, I.,
\textit{Metric entropy of convex hulls in type $p$ spaces--the critical case.}
Proc. Amer. Math. Soc. \textbf{130} (2002), 733-743.

\bibitem{EEH1}
Edmunds, D. E., Evans, W. D. and
Harris, D. J.,
\textit{Approximation numbers of certain Volterra integral operators.}
J. London Math.~Soc. {\bf 37} (1988), 471--489.

\bibitem{EEH2}
Edmunds, D. E., Evans, W. D. and
Harris, D. J.,
\textit{Two--sided estimates of the approximation numbers of certain
Volterra integral operators.}
Studia Math. {\bf 124} (1997), 59--80.

\bibitem{KuLi}
Kuelbs, J. and Li W.~V.,
\textit{Metric entropy and the small ball problem for Gaussian measures.}
\ J. Funct. Anal. {\bf 116} (1993), 133--157.
\bibitem{Led}
Ledoux, M.,
\textit{Isoperimetry and Gaussian analysis.} Lectures on Probability
  Theory and Statistics, Lecture Notes in Math., vol. \textbf{1648}, Springer, 1996,
  pp.~165--294.
\bibitem{LiLin}
 Li W.~V. and Linde, W.,
\textit  {Approximation, metric entropy and small ball estimates for Gaussian measures.}
\ Ann. Probab. {\bf 27} (1999), 1556--1578.


\bibitem{Lif10}
Lifshits, M. A.,
\textit{Bounds for entropy numbers for some critical operators.}
Preprint (2010), {\tt www.arxiv.org/abs/1002.1377}.

\bibitem{LifLi}
Lifshits, M. A. and Linde, W.,
\textit{Approximation and entropy numbers of Volterra operators with application to Brownian motion.}
Mem. Amer. Math. Soc.~\textbf{745} (2002), 87 p.

\bibitem{LLS}
Lifshits, M. A., Linde, W. and Shi, Z.,
 \textit{Small deviations for  Riemann-Liouville
 processes in $L_q$-spaces with respect to fractal measures.}
Proc. London Math. Soc. \textbf{92} (2006), 224--250.

%\bibitem{LLS2}
%Lifshits, M. A., Linde, W. and Shi, Z.,
%\textit{Small deviations of Gaussian random fields in $L_q$--spaces.}
%Electron. J. Probab. \textbf{11} (2006), 1204-1233.

\bibitem{Ma}
Maz'ja, V. G.,
\textit{Sobolev Spaces.}
Springer Verlag Berlin, 1985.

\bibitem{Pem}
Pemantle, R.
\textit {Search cost for a nearly optimal path in a binary tree.}
Ann. Appl. Probab.~{\bf 19} (2009), 1273--1291.

\bibitem{Sch}
Sch\"utt, C.,
\textit{ Entropy numbers of diagonal operators between symmetric
Banach spaces.}
 J. Approx.~Theory {\bf 40} (1984), 121--128.

\bibitem{St}
Steinwart, I.,
\textit{Entropy of $C(K)$-valued operators.}
J. Approx.~Theory \textbf{103} (2000), 302-328.
\bibitem{Ta}
Talagrand, M.,
\textit{New Gaussian estimates for enlarged balls.} Geom. and
  Funct. Anal. \textbf{3} (1993), 502--526.

\bibitem{TJ}
Tomczak--Jaegermann, N.,
\textit{Dualit\'e des nombres d'entropie pour des
op\'erateurs \'a valeurs dans un espace de Hilbert.}
C.R. Acad. Sci.~Paris {\bf 305} (1987), 299-301.
}

\end{thebibliography}

\vspace{1cm}

\parbox[t]{7cm}
{Mikhail Lifshits\\
St.Petersburg State University\\
Dept Math. Mech. \\
198504 Stary Peterhof, \\
Bibliotechnaya pl., 2\\
Russia\\
email: lifts@mail.rcom.ru}\hfill
\parbox[t]{6cm}
{Werner Linde\\
Friedrich--Schiller--Universit\"at Jena \\
Institut f\"ur Stochastik\\
Ernst--Abbe--Platz 2\\
07743 Jena\\
Germany\\
email: werner.linde@uni-jena.de}
\end{document}